\documentclass[11pt,letterpaper]{amsart}

\usepackage{geometry}
\usepackage{amsmath,amssymb,amsthm,mathtools}
\usepackage{booktabs,tabularx}
\usepackage[colorlinks=true,linkcolor=blue,citecolor=blue,urlcolor=blue]{hyperref}
\usepackage{enumitem}

\newtheorem{theorem}{Theorem}[section]
\newtheorem{proposition}[theorem]{Proposition}
\newtheorem{lemma}[theorem]{Lemma}
\newtheorem{corollary}[theorem]{Corollary}
\theoremstyle{definition}
\newtheorem{definition}[theorem]{Definition}
\theoremstyle{remark}
\newtheorem{remark}[theorem]{Remark}

\newcommand{\cX}{\mathcal X}
\newcommand{\cY}{\mathcal Y}

\newcommand{\cL}{\mathcal L}
\newcommand{\cM}{\mathcal M}
\newcommand{\cA}{\mathcal A}
\newcommand{\NA}{\mathrm{NA}}
\newcommand{\Ent}{\operatorname{Ent}}
\newcommand{\tr}{\operatorname{tr}}
\newcommand{\Aut}{\operatorname{Aut}}
\newcommand{\PSH}{\operatorname{PSH}}

\DeclareMathOperator{\DF}{DF}
\newcommand{\ord}{\mathrm{ord}}
\newcommand{\an}{\mathrm{an}}

\title{An algebraic comparison of $\widehat{\mathrm K}$-stability and $\mathrm K^\beta$-stability}
\author[T. S. Papazachariou]{Theodoros Stylianos Papazachariou}
\address{Yau Mathematical Sciences Center, Jingzhai, Tsinghua University, Haidian District, Beijing, China.}
\email{tpapazachariou@mail.tsinghua.edu.cn}

\begin{document}

\begin{abstract}
We establish a uniform quantitative comparison between the non-Archimedean Mabuchi functional and the Darvas--Zhang quantised Mabuchi functional. For a smooth polarised variety with finite automorphism group, we provide an algebraic proof that uniform $\widehat{\mathrm K}$-polystability is equivalent to uniform $\mathrm K^\beta$-stability for every sufficiently large rational $\beta>1$, without passing through the existence of a cscK metric. The main input is a square-root estimate for directional derivatives of non-Archimedean energies, proved by intersection theory. We also study a reduced $\mathrm{K}^\beta$-stability condition when the identity component of the polarised automorphism group is reductive and the Futaki character vanishes. We conclude by characterising K-semistability in terms of asymptotic $\mathrm K^\beta$-semistability and prove the quantitative convergence of the corresponding stability thresholds.
\end{abstract}

\maketitle

\section{Introduction}
One of the main points of focus in complex geometry has been the question of the existence of canonical metrics on manifolds. In particular, let $(X,L)$ be a smooth complex projective polarised manifold. A K\"ahler metric $\omega\in c_1(L)$ whose scalar curvature is constant is called a constant scalar curvature K{\"a}hler (cscK) metric. The existence of cscK metrics was conjectured to be equivalent to an algebrogeometric stability condition called K-polystability \cite{Yau93,Tia97,Don02}. This is now known as the Yau--Tian--Donaldson conjecture. Up until recently, this conjecture was only established in the Fano setting \cite{CDS15a,CDS15b,CDS15c,Tia15}, for polarised Calabi--Yau and canonically polarised varieties \cite{Yau78,Aub78,Oda12}, and for toric surfaces \cite{Don09}. In the general polarised setting, although the existence of a cscK metric is known to imply K-polystability \cite{Don05, Sto09, BDL20}, however, the classical formulation is now known to fail; there exists a smooth polarised projective fivefold which is K-polystable but admits no extremal metric, and in particular no cscK metric \cite[Theorem A]{Liu26}.

This example has come on the heels of the results of \cite{ACGTF08} which showed how potential counterexamples to the YTD conjecture could be constructed (see also \cite{Hat26}).

As such, the expectation had been for a while, that the classical definition of K-stability needs to be strengthened using objects more general than test configurations. This philosophy has been explored in several works \cite{Don12,Sze15,Mab16,Li22,BJ23,DR24}. 

Recently, in a sequence of works, Boucksom--Jonsson introduced $\widehat{\mathrm K}$-stability \cite{BJ23,BJ26}, using non-Archimedean techniques \cite{BHJ17,BHJ19, BJ22,BJ25} and building on previous ideas \cite{BBJ21} relating the existence of metrics via valuative non-Archimedean arguments. Using this approach, they proved that the existence of a cscK metric in $c_1(L)$ is equivalent to $\widehat{\mathrm K}$-polystability \cite[Theorem A]{BJ26}. In addition, the notion of uniform K-stability was developed in \cite{Der16,BHJ17}, as an alternative to K-stability. We refer the reader to the survey \cite{BJ26a} and the references therein for a more detailed exposition.

In a recent result Trusiani solved the Boucksom--Jonsson energy regularisation Conjecture (\cite[Conjectures 4.4, 4.7]{Li21} and \cite[Conjecture 2.5]{BJ18}) via a special Fujita approximation theorem. This allowed him to identify the uniform and $\widehat{\mathrm K}$-polystability notions \cite[Corollary B]{Tru26}. 

In parallel, Darvas--Zhang introduced a different stability condition, called (uniform) $\mathrm K^\beta$-stability using quantised Archimedean versions of the functionals that had appeared in the study of K-stability. Using this stability notion, they were able to establish a Yau--Tian--Donaldson theorem which identifies uniform $\mathrm K^\beta$-stability with the existence of a unique cscK metric; in particular, when $\Aut^0(X,L)$ is trivial, it gives a cscK metric existence criterion \cite[Theorem 1.1]{DZ26}.

Hence, the equivalence
\[
 \text{uniform }\widehat{\mathrm K}\text{-polystability} \quad\Longleftrightarrow\quad \text{uniform }\mathrm K^\beta\text{-stability for some }\beta>1
\]
when $\Aut^0(X,L)=\{1\}$ follows by passing through the existence of a cscK metric. The aim of this note is to prove this equivalence algebraically, without using the existence of cscK metrics. Our first main result is the following.

\begin{theorem}[See Theorem \ref{cor:algebraic}]
\label{thm:intro main}
Assume that $\Aut^0(X,L)=\{1\}$. Then the following conditions are equivalent:
\begin{enumerate}[label=\textup{(\roman*)}]
\item $(X,L)$ is uniformly $\widehat{\mathrm K}$-polystable;
\item $(X,L)$ is uniformly $\mathrm K^\beta$-stable for every sufficiently large rational $\beta$;
\item $(X,L)$ is uniformly $\mathrm K^\beta$-stable for some rational $\beta>1$.
\end{enumerate}
\end{theorem}

Let us now explain how Theorem \ref{thm:intro main} is proved. Recall first that $(X,L)$ is \emph{K-stable} if for any (normal) \emph{ample test configuration} $(\cX,\cL)$, the Donaldson-Futaki invariant $\DF(\cX,\cL)$ is non-negative and equal to zero exactly when $(\cX,\cL) = (X\times\mathbb P^1,p_1^*L+c\cX_0)$ for $c\in \mathbb Q$, i.e. when $(\cX,\cL)$ is trivial. Instead, uniform K-stability holds if $\DF(\cX,\cL)\geq \gamma J^\NA(\cX,\cL)$ for a uniform constant $\gamma$, where $J^\NA(\cX,\cL)$ is called the \emph{non-Archimedean J-functional}, which is equivalent to a minimum norm that measures the triviality of ample test configurations. $(X,L)$ is called \emph{K-polystable} if, for each ample test configuration $(\cX,\cL)$ for $(X,L)$, the Donaldson–Futaki invariant $\DF(\cX,\cL)$ is nonnegative, and vanishes only on product test configurations, arising from cocharacters $\mathbb C^* \rightarrow \Aut(X,L)$; see \cite{Don02} and \cite{Fut83}.

The idea behind $\widehat{\mathrm{K}}$-stability is to enlarge the testing objects of the stability condition, from ample test configurations to the space of their metric completion. The space of ample test configurations is first identified with a space $\mathcal H^\NA$ of Fubini–Study non-Archimedean metrics, whose elements are continuous metrics on $L^{\mathrm{an}}$, where here $L^{\mathrm{an}}$ is the Berkovich analytification of $L$ (see Section \ref{subsec:valuations-analytification}). The elements of $\mathcal H^\NA$ can be also viewed as continuous functions (potentials) on $X^{\mathrm{an}}$, and hence on the set of divisorial valuations. This space is also equipped naturally with a \emph{Darvas metric} $d_1^\NA$ \cite{Dar15}, and its metric completion is identified with the space $\mathcal E^{1,\NA}$ of \emph{potentials with finite energy} \cite{BJ22}.

By the results of \cite[\S7.4, Proposition 8.2 and Remark 8.3]{BHJ17} the Donaldson–Futaki invariant is replaced by the closely related non-Archimedean Mabuchi functional $\mathrm K^\NA:\mathcal H^\NA\rightarrow \mathbb Q$ (following ideas in \cite{Mab86}), which is homogeneous for the action of $\mathbb Q_{>0}$ \cite[Definition 7.13 and Proposition 7.14]{BHJ17}, and coincides with the Donaldson-Futaki invariant on test configurations with reduced central fibre. Via \cite{BJ22,BJ23,BJ25}, this functional admits a natural lower semicontinuous extension $\mathrm K^\NA:\mathcal E^{1,\NA}\rightarrow \mathbb R\cup \{+\infty\}$. $(X,L)$ is \emph{$\widehat{\mathrm{K}}$-polystable} if $\mathrm K^\NA\geq 0$ for all $\varphi\in \mathcal E^{1,\NA}$, with equality only if $\varphi\in \mathcal P_{\mathbb R}$, where $\mathcal P_{\mathbb R}$ is the space of real product test configurations. Similarly $(X,L)$ is \emph{uniform $\widehat{\mathrm{K}}$-polystable} if $\Aut^0(X,L)$ is reductive and there exists $\gamma>0$ such that $K^\NA\geq \gamma d_1^\NA(\varphi, \mathcal P_{\mathbb R})$ for all $\varphi\in \mathcal E^{1,\NA}$ (cf. \cite[\S 1, 8]{BJ26}).

Darvas--Zhang introduce a family of quantised $\mathrm K$-energies $\mathrm K^\beta$, obtained by replacing the entropy term in the Mabuchi functional by its $\beta$-quantised analogue. They then show \cite[Theorem 9.5]{DZ26} that these can be represented via an intersection-theoretic $\beta$-Mabuchi functional $\mathrm{K}^\beta$ that depends on $\beta$.

For a fixed $\beta>0$, they call $(X,L)$ uniformly $\mathrm K^\beta$-stable if there exists $\gamma>0$ such that $\mathrm K^\beta(\varphi_{\mathcal T})\geq\gamma J_{\mathcal T}$ for every ample test configuration $\mathcal T$ \cite[(9)--(10)]{DZ26}.

Our approach in proving Theorem \ref{thm:intro main} requires a direct comparison of $\mathrm K^\NA$ and $\mathrm K^\beta$ for sufficiently large $\beta$. By \cite[Corollary 8.3]{DZ26}, we have
\[
 \mathrm K^\beta(\varphi_{\mathcal T}) \leq\mathrm K^\NA(\varphi_{\mathcal T}).
\]
Hence uniform $\mathrm K^\beta$-stability implies ordinary uniform K-stability. Since $\Aut^0(X,L)$ is trivial, the latter is equivalent to uniform $\widehat{\mathrm K}$-polystability by \cite[Corollary B]{Tru26}. This implies the direction $\text{uniform }\mathrm K^\beta\text{-stability for some }\beta>1 \quad\Longrightarrow\quad\text{uniform }\widehat{\mathrm K}\text{-polystability}$. We thus aim to give a direct proof of the inverse implication.

Starting from a fixed test configuration $\mathcal T$, with associated metric $\varphi_{\mathcal T}\in\mathcal H^\NA$, we take a smooth equivariant model $(\cY,\mathcal M)$ dominating $\mathcal T$ and let $\varphi_\beta$ be the psh envelope of the corresponding log-discrepancy obstacle. The point is that uniform $\widehat{\mathrm K}$-stability gives a lower bound for $\mathrm K^\NA(\varphi_\beta)$, whereas the quantity we ultimately wish to control is $\mathrm K^\beta(\varphi_{\mathcal T})$. We therefore need to control uniformly the loss in passing between these two quantities. The main technical input is the following estimate.

\begin{theorem}[See Theorem \ref{lem:algebraic-polarisation-error}]\label{thm:algebraic-polarisation-error-intro}
For every $\theta\in N^1(X)_{\mathbb R}$ there is a constant $C_\theta>0$, depending only on $(X,L)$ and $\theta$, such that
\begin{equation*}
 |\nabla_\theta E^\NA(\psi)-\nabla_\theta E^\NA(\varphi)|
 \leq C_\theta
 \bigl(E^\NA(\psi)-E^\NA(\varphi)\bigr)^{1/2}
 J^\NA(\varphi)^{1/2}
\end{equation*}
whenever $\varphi$ is a semipositive PL metric, $\varphi\leq\psi$, and $\sup\varphi=\sup\psi=0$. 
\end{theorem}

Here $\nabla_\theta E^\NA(\psi)$ denotes the directional derivative of the non-Archimedean energy with respect to the numerical class $\theta$, as in \cite[Lemma 2.4]{BJ23}. By the positive-intersection formalism of \cite[\S2.1.2]{DR26}, all the terms in the estimate admit intersection-theoretic descriptions. We apply the relative Hodge index theorem on a common dominant model and obtain the displayed bound directly from the resulting intersection numbers. The full argument is detailed in Section \ref{sec: intersection estimates}. 

We note that a similar estimate appears in \cite[Lemma~2.8]{BJ23}, and would also be sufficient for the qualitative results of this paper. Theorem \ref{thm:algebraic-polarisation-error-intro} gives both a sharper bound, and the explicit $O(\beta^{-1/2})$ error used below.

Theorem \ref{thm:algebraic-polarisation-error-intro} combined with a method that uses the results of Trusiani and passes from studying the functionals over $\varphi_\beta$ to the functionals over $\varphi$ (see Proposition \ref{prop:input} and Theorem \ref{thm:mabuchi-comparison}), implies that there exists $\varepsilon_\beta\geq0$, with $\varepsilon_\beta\to0$, such that
\[
 \mathrm K^\beta(\varphi_{\mathcal T}) \geq \bigl(\Lambda-\varepsilon_\beta\bigr) \|\varphi_{\mathcal T}\|_1;
\]
(see Corollary \ref{cor:transfer}). Consequently, if $\Lambda>0$, then for every sufficiently large rational $\beta$,
\[
 \mathrm K^\beta(\varphi_{\mathcal T}) \geq\frac{\Lambda}{2}\|\varphi_{\mathcal T}\|_1.
\]
These estimates allow us to prove the direction
\[\text{uniform }\widehat{\mathrm K}\text{-polystability} \quad\Longrightarrow\quad \text{uniform }\mathrm K^\beta\text{-stability for some }\beta>1\]
in Theorem \ref{thm:intro main} algebraically and without assuming the existence of a cscK metric.

When $\Aut^0(X,L)$ is reductive, we also study a reduced fixed-$\mathrm{K}^\beta$-stability condition. Reductivity here is a harmless assumption since the existence of cscK metrics on smooth manifolds imply the reductivity of their automorphism groups \cite{Lic57, Mat57}. Fixing a maximal torus $\mathbb T$, we consider $\mathbb T$-equivariant normal ample test configurations. We let $N_{\mathbb Q}:=\operatorname{Hom}(\mathbb G_m,\mathbb T)\otimes_{\mathbb Z}\mathbb Q$ and $N_{\mathbb R}:=N_{\mathbb Q}\otimes_{\mathbb Q}\mathbb R$. A cocharacter $\xi\in N_{\mathbb Q}$ defines a rational product twist $\mathcal T_\xi$ of $\mathcal T$, whose associated metric is denoted by $\xi\star\varphi_{\mathcal T}$; the twisting action extends continuously to $N_{\mathbb R}$, and we set
\[
 q_{\mathbb T}(\varphi):=\inf_{\xi\in N_{\mathbb R}}\|\xi\star\varphi\|_1.
\]
We define
\[\mathrm K_{\mathbb T,\mathrm{red}}^\beta (\varphi_{\mathcal T}) :=\sup_{\xi\in N_{\mathbb Q}} \mathrm K^\beta(\varphi_{\mathcal T_\xi}).\]
This condition is closely related to, but not identical with, the reduced formalism of \cite[\S10, Equations (86), (89) and Theorem 10.4]{DZ26}. The main difference is that Darvas--Zhang use the reduced radial $J$-functional and impose a $G$-uniform Mabuchi inequality on log-discrepancy models, whereas we fix $\beta$, take the supremum of the $\beta$-Mabuchi invariant over rational product twists of an ample test configuration, and compare the result with $q_{\mathbb T}$. However, due to the Theorem below and \cite[Theorem 1.6]{DZ26}, \cite[Theorem A]{BJ26}, we see that the two notions coincide.

\begin{theorem}[See Theorem \ref{thm:reduced-Kbeta-comparison}]\label{thm:intro_K-ps}
Suppose that $G=\Aut^0(X,L)$ is reductive and that $\operatorname{Fut}_{X,L}=0$. Fix a maximal algebraic torus $\mathbb T\subset G$. Then the following conditions are equivalent:
\begin{enumerate}[label=\textup{(\roman*)}]
\item $(X,L)$ is uniformly $\widehat{\mathrm K}$-polystable;
\item $(X,L)$ is uniformly reduced $\mathrm K^\beta$-polystable for every sufficiently large rational $\beta$;
\item $(X,L)$ is uniformly reduced $\mathrm K^\beta$-polystable for some rational $\beta>1$.
\end{enumerate}
\end{theorem}

Finally, we compare the corresponding semistability notions. For every normal ample test configuration $\mathcal T$, we define
\[
 \lambda_\beta(X,L) :=\inf_{\|\varphi_{\mathcal T}\|_1>0} \frac{\mathrm K^\beta(\varphi_{\mathcal T})}{\|\varphi_{\mathcal T}\|_1}, \qquad \lambda_\infty(X,L) :=\inf_{\|\varphi_{\mathcal T}\|_1>0} \frac{\mathrm K^\NA(\varphi_{\mathcal T})}{\|\varphi_{\mathcal T}\|_1}.
\]
Notice that K-semistability is equivalent to $\lambda_\infty\geq 0$. Using our results from Section \ref{sec: mabuchi comparison} we are able to deduce the following numerical estimate.

\begin{theorem}\label{thm:intro-threshold-convergence}
Suppose that $(X,L)$ is K-semistable. Then there exists $\varepsilon_\beta\geq 0$, with $\varepsilon_\beta\to 0$ as $\beta \to \infty$ such that
\begin{equation*}
 \lambda_\infty-\varepsilon_\beta \leq\lambda_\beta\leq\lambda_\infty.
\end{equation*}
In particular, $\lambda_\beta(X,L)\nearrow\lambda_\infty(X,L)$ as $\beta\to\infty$.
\end{theorem}

We thus define an asymptotic $\mathrm K^\beta$ stability notion. We say $(X,L)$ is asymptotically $\mathrm K^\beta$-semistable if $\liminf_{\beta\to\infty,\,\beta\in\mathbb Q}\lambda_\beta(X,L)\geq0$. The preceding theorem gives the following characterisation.

\begin{corollary}[See Corollary \ref{cor:asymptotic-semistability}(i)]\label{thm:intro K-ss}
Let $(X,L)$ be a smooth polarised variety. Then $(X,L)$ is K-semistable if and only if $(X,L)$ is asymptotically $\mathrm K^\beta$-semistable. 
\end{corollary}

We note that the above results can also apply to log canonical varieties that satisfy the envelope property, as the main results require only that $A_X\geq 0$, and that the psh envelope is genuinely psh. In this setting the notion of $\mathrm{K}^\beta$-stability extends naturally using the intersection theoretic definition in \cite[Theorem 9.5]{DZ26}. However, we will restrict ourselves to the case where $X$ is smooth, noting that the extension to the more general setting is immediate and follows the same arguments.

\subsection*{Structure of the paper}
In Section \ref{sec:NA-background} we recall the non-Archimedean and intersection-theoretic notions used in the paper, together with the intersection-theoretic formula for $\mathrm K^\beta$. In Section \ref{sec: mixed holder} we construct the log-discrepancy metric, prove the estimate of Theorem \ref{thm:algebraic-polarisation-error-intro}, and deduce the quantitative comparison of the Mabuchi functionals. In Section \ref{sec:stability-consequences} we apply this comparison to uniform stability, its equivariant version, and the convergence of the stability thresholds.

\subsection*{Acknowledgments}
I would like to thank Ruadha\'i Dervan and Antonio Trusiani for many helpful conversations and valuable comments. I am supported by the Beijing Natural Science Foundation Project IS25037 and a Shuimu Scholar Programme Scholarship at Tsinghua University.

\section{Non-Archimedean preliminaries and stability conventions}
\label{sec:NA-background}

Throughout, $X$ is a smooth complex projective variety of dimension $n$, $L$ is an ample $\mathbb Q$-line bundle, and $V=(L^n)$. All test configurations are normal, relatively ample, and compactified over $\mathbb P^1$, unless a big model is explicitly mentioned. Our non-Archimedean pluripotential conventions follow \cite[\S\S4--5]{BJ25}, \cite[\S\S2,4--5]{BJ23}, and \cite[\S\S4--8]{BJ26} as well as \cite{BJ22}. The Donaldson--Futaki conventions follow \cite{BHJ17}, the positive-intersection functionals on big models follow \cite{DR26}, and the intersection-theoretic formula for $\mathrm K^\beta$ follows \cite{DZ26}.

A $\mathbb Q$-line bundle means an element of $\operatorname{Pic}(X)\otimes_{\mathbb Z}\mathbb Q$. We write $N^1(X)_{\mathbb R}$ for numerical divisor classes. For an ample class $\omega=c_1(L)$, the notation $V=(\omega^n)=(L^n)$ always denotes its top self-intersection.

\subsection{Valuations and the Berkovich analytification} \label{subsec:valuations-analytification}

The Berkovich analytification $X^{\mathrm{an}}$ with respect to the trivial absolute value is a compact Hausdorff space whose points may be viewed as real semivaluations on $X$. For a local section $s$ of a line bundle, set $|s|(v):=e^{-v(s)}$ after choosing a local trivialisation; see \cite[\S\S1.5--1.6]{BJ25}.

The \emph{trivial valuation} $v_{\mathrm{triv}}$ is zero on every nonzero rational function. It is the distinguished base point of $X^{\mathrm{an}}$.

A \emph{divisorial valuation} is $v=c\,\operatorname{ord}_E$, where $c\in\mathbb R_{>0}$ and where $E$ is a prime divisor on a normal birational model $Y\to X$. The set of divisorial valuations is denoted $X^{\mathrm{div}}$; it is dense
in $X^{\mathrm{an}}$.

Its \emph{log discrepancy} is
\[
 A_X(c\operatorname{ord}_E) :=c\bigl(1+\operatorname{ord}_E(K_{Y/X})\bigr).
\]
Since $X$ is smooth, $A_X\geq0$ on $X^{\mathrm{div}}$. We use its greatest lower-semicontinuous extension
\[
 A_X:X^{\mathrm{an}}\longrightarrow[0,+\infty]
\]
see \cite[Definition A.5 and Theorem A.10]{BJ23} and also \cite[\S 2.3]{DP24}.

\subsection{Test configurations and metric spaces}\label{subsec:TC-filtration-norm}

We now recall the definition of test configurations.

\begin{definition}\label{def:TC-background}
A \emph{normal test configuration} for $(X,L)$ is a normal variety $\cX$ equipped with a $\mathbb G_m$-equivariant proper flat map $\pi:\cX\to\mathbb A^1$ and a $\mathbb G_m$-linearised
$\mathbb Q$-line bundle $\cL$ such that
\[
 (\cX,\cL)|_{\mathbb A^1\setminus\{0\}}
 \simeq (X,L)\times(\mathbb A^1\setminus\{0\})
\]
$\mathbb G_m$-equivariantly. It is ample, semiample, or nef when $\cL$ has the corresponding relative positivity property. The same symbols $(\cX,\cL)$ denote the canonical compactification over $\mathbb P^1$.
\end{definition}

We can write the central fibre as a Weil divisor $\cX_0=\sum_E b_EE$. Then, the generic-fibre identification identifies $\mathbb C(X)$ with a subfield of $\mathbb C(\cX)$, and each irreducible component $E$ determines the divisorial valuation $v_E:=b_E^{-1}\operatorname{ord}_E\big|_{\mathbb C(X)}$. The factor $b_E^{-1}$ compensates for the multiplicity of the central fibre. We call the $v_E$ the \emph{Rees valuations} of the test configuration.

Two test configurations are called equivalent if their pullbacks agree on a common equivariant model. 

\subsection{Fubini--Study metrics and algebraic non-Archimedean metrics}\label{subsec:FS-metrics}

A \emph{metric} on the analytification $L^{\mathrm{an}}$ assigns a norm to each one-dimensional fibre, varying over $X^{\mathrm{an}}$. After fixing the trivial reference metric, it is encoded by its potential $\varphi:X^{\mathrm{an}}\to\mathbb R$. In particular, for a local section $s$ of $mL$,
\[\|s(v)\|_{m\varphi}=|s|(v)e^{-m\varphi(v)}.\]

\begin{definition} \label{def:metric-order-background}
Let $\varphi$ and $\psi$ be two metrics on the same line bundle $L$, viewed as potentials. We write
\begin{equation}\label{eq:metric-pointwise-order}
 \varphi\leq\psi \quad\Longleftrightarrow\quad \varphi(v)\leq\psi(v).
\end{equation}
for every $v\in X^{\mathrm{an}}$.
\end{definition}
Equivalently, for a local nonzero section $s$ of $mL$ we also obtain
\[\varphi\leq\psi \quad\Longleftrightarrow\quad \|s(v)\|_{m\varphi}\geq\|s(v)\|_{m\psi}.\]

A real \emph{Fubini--Study function} on $L$ has the form
\[\varphi(v)=\frac1m\max_j\{\log|s_j|(v)+\lambda_j\},\]
where the sections $s_j\in H^0(X,mL)$ have no common zero and $\lambda_j\in\mathbb R$. It is called \emph{rational} if all $\lambda_j$ are rational. A rational Fubini--Study metric is also called an \emph{algebraic semipositive metric}; a real Fubini--Study metric is its finite-rank real-weight analogue.

We denote the spaces of real and rational Fubini--Study metrics by $\mathcal H_{\mathbb R}^\NA(L)$ and $\mathcal H_{\mathbb Q}^\NA(L)$. They are stable under maxima and translation by constants.

Note that an ample test configuration $\mathcal T=(\cX,\cL)$ determines a non-Archimedean metric $\varphi_{\mathcal T}\in\mathcal H^\NA(L)$. Under the Rees correspondence, the potential attached to an ample test configuration is exactly its rational Fubini--Study metric. The scaling action on potentials is $(t\cdot\varphi)(v):=t\varphi(t^{-1}v),$ where $t>0$. For a test configuration $(\cX,\cL)$, replacing $\cL$ by $\cL+c\cX_0$ translates the potential by $c$.

More explicitly, we choose an equivariant normal model $\mathcal Y$ dominating both $\mathcal X$ and $X\times\mathbb P^1$, with morphisms $\nu:\mathcal Y\to\mathcal X$ and $\rho:\mathcal Y\to X\times\mathbb P^1$, and write $\nu^*\mathcal L-\rho^*p_1^*L=D$, where $D$ is a vertical $\mathbb Q$-Cartier divisor. If $E$ is an irreducible component of $\mathcal Y_0$, $b_E=\ord_E(\mathcal Y_0)$, and $v_E=b_E^{-1}\ord_E|_{\mathbb C(X)}$, then
\[
 \varphi_{\mathcal T}(v_E) =b_E^{-1}\ord_E(D).
\]
see \cite[(A.5) and Lemma A.12]{BJ25}.

In the rest of the paper, an \emph{ample test-configuration metric} means an element of $\mathcal H_{\mathbb Q}^\NA(L)$ arising from a normal ample test configuration. We will usually drop the subscript and write $\mathcal H^\NA(L):=\mathcal H_{\mathbb Q}^\NA(L)$.

\subsection{Psh envelopes}

An $L$-plurisubharmonic function is an upper-semicontinuous function $\varphi:X^{\mathrm{an}}\to\mathbb R\cup\{-\infty\}$, not identically $-\infty$, that is the pointwise limit of a decreasing sequence or net of real Fubini--Study metrics. The space is denoted $\PSH^\NA(L)$. It is convex, stable under addition of constants and finite maxima, and each of its elements is finite on $X^{\mathrm{div}}$. Moreover, an $L$-psh function is uniquely determined by its restriction to $X^{\mathrm{div}}$. We give $\PSH^\NA(L)$ the weak topology of pointwise convergence on $X^{\mathrm{div}}$.

The uniformly closed subclass
\[\operatorname{CPSH}^\NA(L) :=\PSH^\NA(L)\cap C^0(X^{\mathrm{an}})\]
is the space of continuous semipositive psh metrics. Equivalently, it is the uniform closure of the real Fubini--Study metrics. Thus
\[\mathcal H_{\mathbb Q}^\NA(L) \subset\mathcal H_{\mathbb R}^\NA(L) \subset\operatorname{CPSH}^\NA(L) \subset\PSH^\NA(L).\]

For psh metrics it is enough to check the inequality \eqref{eq:metric-pointwise-order} on $X^{\mathrm{div}}$, because a psh metric is determined by its values at divisorial valuations. Thus every ordered pair later in the manuscript is ordered as a pair of \emph{potentials}.

\begin{definition}\label{def: envelopes}
An \emph{obstacle} is an extended-real-valued function that prescribes an upper bound in a family of psh metrics; it need not itself be psh. If $f$ is such an obstacle and the set displayed below is nonempty, its \emph{psh envelope} is
\[P(f):=\left( \sup\left\{\eta\in\PSH^\NA(L)\ \middle|\ \eta(v)\leq f(v)\text{ for all }v\in X^{\mathrm{div}}\right\}\right)^*.\]
The star denotes upper-semicontinuous regularisation. See \cite[\S4.6]{BJ26}.
\end{definition}

The pair $(X,L)$ has the \emph{envelope property} when the supremum defining $P(f)$ is already upper semicontinuous, hence $L$-psh, for every continuous obstacle $f$. This property holds when $X$ is smooth, by the multiplier-ideal argument in non-Archimedean pluripotential theory; see \cite[\S4.4]{BJ25}.

\subsection{Energy, distance, and norms}

We will denote by $E^\NA$, $\operatorname{MA}$, $d_1^\NA$, and $J^\NA$ the \emph{Monge--Amp\`ere energy}, \emph{Monge--Amp\`ere measure}, \emph{Darvas metric}, and \emph{$J$-functional}. Let us explain these notions in more detail.

For an ample or nef test configuration metric $\varphi_{\cX,\cL}$ represented on the canonical compactification $(\cX,\cL)$, we define
\[E^\NA(\varphi_{\cX,\cL}) :=\frac{(\cL^{n+1})}{(n+1)V}.\]
This defines the non-Archimedean Monge--Amp\`ere energy on rational Fubini--Study metrics.

The energy extends first by uniform continuity to $\operatorname{CPSH}^\NA(L)$ and then uniquely to an increasing, upper-semicontinuous functional
\[E^\NA:\PSH^\NA(L)\longrightarrow\mathbb R\cup\{-\infty\}, \qquad E^\NA(\varphi) =\inf_{\substack{\psi\in\operatorname{CPSH}^\NA(L)\\ \psi\geq\varphi}} E^\NA(\psi),\]
which is concave, monotone, and translation equivariant (cf. \cite[\S4.2, Equations (4.2)--(4.5), and (4.8)]{BJ25}):
\begin{equation*}
 \varphi\leq\psi\Rightarrow E^\NA(\varphi)\leq E^\NA(\psi), \qquad E^\NA(\varphi+c)=E^\NA(\varphi)+c.
\end{equation*}
We thus obtain the following finite-energy space of metrics
\[\mathcal E^{1,\NA}(L) :=\{\varphi\in\PSH^\NA(L)\mid E^\NA(\varphi)>-\infty\}.\]
The complete hierarchy of metric classes is
\begin{equation}\label{eq:metric-class-inclusions}
 \mathcal H_{\mathbb Q}^\NA(L) \subset\mathcal H_{\mathbb R}^\NA(L) \subset\operatorname{CPSH}^\NA(L) \subset\mathcal E^{1,\NA}(L) \subset\PSH^\NA(L).
\end{equation}

If $\varphi = \varphi_{\cX,\cL}$ for a relatively nef normal test configuration $(\cX,\cL)$ and $\cX_0=\sum_Eb_EE$, its \emph{Monge--Amp\`ere measure} is
\[\operatorname{MA}(\varphi) =\frac1V\sum_E b_E(\cL^n\cdot E)\,\delta_{v_E}.\]
The coefficients are nonnegative and sum to one. The operator $\operatorname{MA}$ is unchanged by adding constants, strongly continuous, and is the differential of the energy:
\[\left.\frac d{dt}\right|_{t=0} E^\NA\bigl((1-t)\varphi+t\psi\bigr) =\int_{X^{\mathrm{an}}}(\psi-\varphi)\, \operatorname{MA}(\varphi);\]
see \cite[(4.2) and \S4.3]{BJ25}.

For $\varphi,\psi\in\mathcal E^{1,\NA}(L)$, the rooftop envelope $P(\varphi\wedge\psi)$ is defined by decreasing approximation from continuous semipositive metrics. The $d_1^\NA$ metric is then characterised by
\begin{equation}\label{eq:d1-rooftop-background}
 d_1^\NA(\varphi,\psi) =E^\NA(\varphi)+E^\NA(\psi) -2E^\NA\bigl(P(\varphi\wedge\psi)\bigr);
\end{equation}
see \cite[Theorem 5.5(i),(iii)]{BJ25}.

On real Fubini--Study metrics this agrees with the spectral distance for $p=1$. By \cite[Theorem 5.5(ii) and Theorem B]{BJ25} it extends uniquely to a metric inducing the strong topology on $\mathcal E^{1,\NA}(L)$. Since $X$ is smooth and hence has the envelope property, $\mathcal E^{1,\NA}(L)$ is complete and both $\mathcal H_{\mathbb R}^\NA(L)$ and $\mathcal H_{\mathbb Q}^\NA(L)$ are dense in it for $d_1^\NA$. In particular,
\[\varphi\leq\psi \quad\Longrightarrow\quad d_1^\NA(\varphi,\psi)=E^\NA(\psi)-E^\NA(\varphi).\]
It is invariant under simultaneous translation, but $d_1^\NA(\varphi+c,\psi)$ depends on $c$. We therefore use the quotient metric 
\[\underline{d}_1^\NA([\varphi],[\psi])
 :=\inf_{c\in\mathbb R}d_1^\NA(\varphi+c,\psi)\]
on $\mathcal E^{1,\NA}(L)/\mathbb R$, and set
\[\|\varphi\|_1 :=\underline{d}_1^\NA([\varphi],[0])
 =\inf_{c\in\mathbb R}d_1^\NA(\varphi+c,0);\]
see \cite[Equation (5.10)]{BJ25}. Since the translation action is isometric for $d_1^\NA$, we have $d_1^\NA(\varphi+c,0)=d_1^\NA(\varphi,-c)$. Replacing $c$ by $-c$ therefore gives
\[
 \|\varphi\|_1 =\inf_{c\in\mathbb R}d_1^\NA(\varphi,c);
\]
see \cite[\S5.4, Equation (5.10)]{BJ25}.

Hence, the $d_1^\NA$-completion of $\mathcal H^\NA(L)$ is the finite-energy space $\mathcal E^{1,\NA}(L)$. Addition of constants acts on both spaces and leaves all stability numerators unchanged.

\begin{lemma}\label{lem:distance-to-fixed-subset}
We have,
\[\bigl|\|\varphi\|_1-\|\psi\|_1\bigr|
 \leq d_1^\NA(\varphi,\psi).\]
\end{lemma}

\begin{proof}
Let $A:=\{c:c\in\mathbb R\}\subset\mathcal E^{1,\NA}(L)$ be the set of constant metrics. For every $c\in\mathbb R$, the triangle inequality gives
\[
 d_1^\NA(\varphi,c) \leq d_1^\NA(\varphi,\psi)+d_1^\NA(\psi,c).
\]
Since the first term on the right is independent of $c$, taking the infimum over $c\in\mathbb R$ yields
\[
 \|\varphi\|_1 \leq d_1^\NA(\varphi,\psi)+\|\psi\|_1.
\]
Interchanging $\varphi$ and $\psi$ gives the reverse inequality and proves the claim.
\end{proof}

\begin{lemma}\label{lem:zero-norm}
We have $\|\varphi\|_1=0$ if and only if $\varphi$ is constant.
\end{lemma}

The $I$-functional is
\[I_\omega(\varphi,\psi) :=\int_{X^{\mathrm{an}}}(\varphi-\psi)\,
 \operatorname{MA}(\psi) -\int_{X^{\mathrm{an}}}(\varphi-\psi)\, \operatorname{MA}(\varphi)\geq0.\]
It is symmetric and translation invariant, but satisfies only a quasi-triangle inequality. We abbreviate $I_\omega(\varphi):=I_\omega(\varphi,0)$.

The \emph{non-Archimedean $J$-functional} is
\[J^\NA(\varphi)
 :=\sup_{X^{\mathrm{an}}}\varphi-E^\NA(\varphi);\]
see \cite[Definition 7.6, Lemma 7.7 and Proposition 7.8]{BHJ17} and \cite[(2.6)]{BJ25}.

The following Lemma is a useful norm comparison.

\begin{lemma}\label{lem:norm-comparison}
There are constants $c_J,C_J>0$, depending only on the dimension and the fixed normalisations, such that for $\varphi\in\mathcal E^{1,\NA}(L)$ we have
\[c_JJ^\NA(\varphi)\leq\|\varphi\|_1
 \leq C_JJ^\NA(\varphi).\]
In particular,
\[\|\varphi_{\mathcal T}\|_1
 \geq c_JJ^\NA(\varphi_{\mathcal T}).\]
\end{lemma}

\begin{proof}
Both $J^\NA$ and $\|\cdot\|_1$ are invariant under addition of constants. By the last assertion of \cite[Lemma 5.12]{BJ25}, there is a constant $C_n>0$, depending only on $n$, such that $I_\omega(\varphi)\leq C_n\|\varphi\|_1$ for every $\varphi\in\mathcal E^{1,\NA}(L)$.

Since $\operatorname{MA}(0)=\delta_{v_{\mathrm{triv}}}$ and $\varphi(v_{\mathrm{triv}})=\sup_{X^{\an}}\varphi$, as recalled in \cite[\S\S4.1--4.2]{BJ25}, we have
\[
 \begin{aligned}
 I_\omega(\varphi) &=\sup_{X^{\an}}\varphi -\int_{X^{\an}}\varphi\,\operatorname{MA}(\varphi)\\
 &=J^\NA(\varphi)+E^\NA(\varphi) -\int_{X^{\an}}\varphi\,\operatorname{MA}(\varphi).
 \end{aligned}
\]
By \cite[Inequality (4.5)]{BJ25}, applied to the pair $(0,\varphi)$, we get
\[
 E^\NA(\varphi) \geq\int_{X^{\an}}\varphi\,\operatorname{MA}(\varphi).
\]
Consequently,
\[
 J^\NA(\varphi) \leq I_\omega(\varphi) \leq C_n\|\varphi\|_1,
\]
which proves the lower bound with $c_J=C_n^{-1}$.

For the other direction, we set $\widetilde\varphi:=\varphi-\sup_{X^{\mathrm{an}}}\varphi$. Both $\|\cdot\|_1$ and $J^\NA$ are translation invariant, so $\|\widetilde\varphi\|_1=\|\varphi\|_1$ and $J^\NA(\widetilde\varphi)=J^\NA(\varphi)$. Since $\sup\widetilde\varphi=0$, we have $\widetilde\varphi\leq0$, and hence $P(\widetilde\varphi\wedge0)=\widetilde\varphi$. As $X$ has the envelope property, the rooftop formula \cite[Theorem 5.5(i),(iii)]{BJ25} applies to finite-energy metrics and gives
\begin{align*}
 d_1^\NA(\widetilde\varphi,0)
 &=E^\NA(\widetilde\varphi)+E^\NA(0) -2E^\NA\bigl(P(\widetilde\varphi\wedge0)\bigr)\\
 &=-E^\NA(\widetilde\varphi) =J^\NA(\widetilde\varphi),
\end{align*}
where we used $E^\NA(0)=0$ and $J^\NA(\widetilde\varphi) =\sup\widetilde\varphi-E^\NA(\widetilde\varphi)$. Therefore
\[
 \|\varphi\|_1
 =\inf_{c\in\mathbb R}d_1^\NA(\widetilde\varphi,c)
 \leq d_1^\NA(\widetilde\varphi,0)
 =J^\NA(\varphi),
\]
which proves the upper bound with $C_J=1$.
\end{proof}

For $\eta\in N^1(X)_{\mathbb R}$, the energy depends smoothly on the polarisation as long as $\omega+t\eta$ remains ample. When the same model function is interpreted relative to these nearby polarisations, its derivative at the fixed class $\omega$ is
\[\nabla_\eta E^\NA_\omega(\varphi):=\left.\frac d{dt}\right|_{t=0} E^\NA_{\omega+t\eta}(\varphi).\]
The derivative is linear in $\eta$ and translation invariant in $\varphi$.
We let $\tr_\omega(\eta):=\frac{n(\eta\cdot\omega^{n-1})}{V}$
and, if $\varphi$ is represented on a model $\cY$ by a nef class $\cL_\varphi$, denote by $\eta_{\cY}$ the horizontal pullback of $\eta$. By \cite[Lemma 2.4 and (2.6)]{BJ23},
\[
 \nabla_\eta E_\omega^\NA(\varphi)
 =\frac1V\bigl(\eta_{\cY}\cdot\cL_\varphi^n\bigr)
 -\tr_\omega(\eta)E_\omega^\NA(\varphi).
\]

\subsection{\texorpdfstring{$\widehat{\mathrm K}$}{K-hat} stability}\label{sec: k-hat}

Let $A_X$ be the log-discrepancy function. For
$\varphi\in\mathcal E^{1,\NA}(L)$ we define the \emph{non-Archimedean entropy}
\[
 \Ent^\NA(\varphi)
 :=\int_{X^{\mathrm{an}}}A_X\,\operatorname{MA}(\varphi).
\]
If $\theta=c_1(-K_X)$, recall that $\nabla_{-K_X}E^\NA=\nabla_\theta E^\NA=-\nabla_{K_X}E^\NA$; see \cite[Lemma 2.4 and (2.6)]{BJ23}. We thus define the \emph{non-Archimedean Mabuchi functional} following \cite[(4.4)--(4.5)]{BJ23} as 
\[\mathrm{K}^\NA(\varphi):=\Ent^\NA(\varphi)+\nabla_{K_X}E^\NA(\varphi)
 =\Ent^\NA(\varphi)-\nabla_{-K_X}E^\NA(\varphi).\]
The entropy may equal $+\infty$, while $\nabla_{K_X}E^\NA$ is finite-valued and strongly continuous.

\begin{remark}\label{rem: DF vs mabuchi}
For a normal ample test configuration $\mathcal T$,
\[
 \DF(\mathcal T) =\mathrm K^\NA(\varphi_{\mathcal T}) +\frac{\bigl((\mathcal X_0-(\mathcal X_0)_{\mathrm{red}}) \cdot\mathcal L^n\bigr)}{V}.
\]
In particular, $\DF(\mathcal T)\geq\mathrm K^\NA(\varphi_{\mathcal T})$.
By \cite[Proposition 8.2 and Remark 8.3]{BHJ17}, K-semistability and uniform K-stability are equivalently expressed as $\mathrm K^\NA(\varphi)\geq0$ and $\mathrm K^\NA(\varphi)\geq\delta J^\NA(\varphi)$, respectively, on ample test-configuration metrics, for some $\delta>0$ in the uniform case. Replacing $J^\NA$ by an equivalent norm changes the optimal numerical constant.
\end{remark}

Let $\mathcal P_{\mathbb R}\subset\mathcal E^{1,\NA}(L)$ denote the closed locus of real product metrics of \cite[Definition 4.12 and Corollary 4.14]{BJ26}. For any nonempty subset $\mathcal A\subset\mathcal E^{1,\NA}(L)$, we write
\[
 d_1^\NA(\varphi,\mathcal A)
 :=\inf_{\psi\in\mathcal A}d_1^\NA(\varphi,\psi).
\]

The following is one of the key notions of this paper defined in \cite[Definition 8.11 and Remark 8.12]{BJ26}.

\begin{definition}\label{def: khat stability}
Put $G:=\Aut^0(X,L)$.
The pair $(X,L)$ is $\widehat{\mathrm K}$-semistable if $\mathrm K^\NA(\varphi)\geq0$ for every $\varphi\in\mathcal E^{1,\NA}(L)$.

The pair $(X,L)$ is $\widehat{\mathrm K}$-polystable if it is $\widehat{\mathrm K}$-semistable and $\mathrm K^\NA(\varphi)=0$ if and only if $\varphi\in\mathcal P_{\mathbb R}$. It is uniformly $\widehat{\mathrm K}$-polystable if $G$ is reductive and, for some $\sigma>0$,
\begin{equation}\label{eq:Kb-uniform}
 \mathrm{K}^\NA(\varphi) \geq\sigma d_1^\NA(\varphi,\mathcal P_{\mathbb R})
\end{equation}
for all $\varphi\in\mathcal E^{1,\NA}(L)$. If $G$ is trivial, the product locus consists of constants and the right-hand side is $\sigma\|\varphi\|_1$.
\end{definition}

\subsection{Big models, model functions, and psh envelopes}
\label{subsec:big-model-background}

\begin{definition}
 A \emph{model over $X\times\mathbb P^1$} is a normal variety $\cX$ with a projective birational morphism $\rho:\cX\to X\times\mathbb P^1$ that is an isomorphism away from the central fibre and is compatible with the $\mathbb G_m$-action.
\end{definition}
Every test configuration has a canonical equivariant rational map to $X\times \mathbb P^1$. By normalising the graph, and then resolving if necessary, it admits a representative dominating both the original test configuration and $X\times \mathbb P^1$. The original test configuration is itself a model over $X\times \mathbb P^1$ only when this rational map extends to a projective birational morphism. Conversely, a model $\cX\to X\times \mathbb P^1$ can be equipped with a relatively ample $\mathbb Q$-line bundle extending $L$, and hence underlies a test configuration of $(X,L)$. However, a given line class $\cL$ on $\cX$ extending $L$ may be merely relatively big rather than relatively ample; in this case the pair $(\cX,\cL)$ is a big model but not an ample test configuration.

\begin{definition}
 A \emph{ dominant big model} for $(X,L)$ is a normal $\mathbb G_m$-equivariant model $\rho:\cX\to X\times\mathbb P^1$ together with a relatively big $\mathbb Q$-line class $\cL$ extending $L$.
\end{definition}

It defines a model function $\varphi_{\cL}$ and the psh metric
\[\varphi_{\cX,\cL}:=P(\varphi_{\cL})
 \in\mathcal E^{1,\NA}(L).\]
Note that the notation is similar to the Fubini--Study metric defined before, but they only coincide if $\cL$ is (semi)-ample. If $\cL$ is merely big, positive intersection products replace ordinary intersection numbers.

On a dominant model $\rho:\cX\to X\times\mathbb P^1$, we use the relative logarithmic canonical class
\begin{equation}\label{eq:relative-log-canonical-background}
 K^{\log}_{\cX/X\times\mathbb P^1}
 :=K_{\cX/X\times\mathbb P^1}
 +(\cX_0)_{\mathrm{red}}-\cX_0
 =K_{\cX}+(\cX_0)_{\mathrm{red}}
 -\rho^*\bigl(K_{X\times\mathbb P^1}
 +X\times\{0\}\bigr).
\end{equation}
\subsection{Positive intersections on big models}\label{subsec:DR-positive-intersection}
For a big $\mathbb R$-class $\alpha$ on a projective variety $Y$, $\langle\alpha^k\rangle$ denotes its positive intersection product, namely the numerical class on the Riemann--Zariski space obtained as the supremum of products of nef approximants to pullbacks of $\alpha$ on higher birational models. If $D$ is an $\mathbb R$-Cartier divisor, $\langle\alpha^{\dim Y-1}\rangle\cdot D$ denotes the resulting pairing with $D$. These expressions recover the ordinary intersection numbers when $\alpha$ is nef, while $\langle\alpha^{\dim Y}\rangle=\operatorname{vol}(\alpha)$; see \cite[\S\S2.1.2--2.1.4, Equations (2.1), (2.2), and (2.7)]{DR26}.

\subsection{\texorpdfstring{$\mathrm{K}^\beta$}{K-beta}-stability} \label{sec:Kbeta-dictionary}

We only consider rational $\beta>1$. Let $\mathcal T=(\cX,\cL)$ be a normal ample test configuration with associated metric $\varphi_{\mathcal T}$. We choose a smooth equivariant representative $(\cY,\cM)$ dominating both $\mathcal T$ and $X\times\mathbb P^1$, and put $D:=K^{\log}_{\cY/X\times\mathbb P^1}$, $\cM_\beta:=\cM+\beta^{-1}D$. We also define the \emph{$\beta$-entropy} $\Ent_{\mathcal T}^\beta $ following \cite[Theorem 9.5]{DZ26} as
\[
 \Ent_{\mathcal T}^\beta :=\frac{\beta}{(n+1)V} \left(\langle\cM_\beta^{n+1}\rangle-\cM^{n+1}\right).
\]
Similarly we define the \emph{$\beta$-Mabuchi functional}
\[
 \mathrm K^\beta(\varphi_{\mathcal T}) :=\Ent_{\mathcal T}^\beta +\nabla_{K_X}E^\NA(\varphi_{\mathcal T}).
\]
We note that these are not the original definitions as in \cite{DZ26}, but their Theorem 9.5 shows that for a rational $\beta>1$ they agree with their radial invariant and are independent of the chosen dominant model.

\begin{definition}[{cf. \cite[p. 3, (10)]{DZ26}}]
\label{def: Kbeta stability}
We fix $\beta\in\mathbb Q_{>1}$. We say that $(X,L)$ is \emph{uniformly $\mathrm K^\beta$-stable} if there exists a constant $\gamma=\gamma_\beta>0$, independent of the test configuration, such that $\mathrm K^\beta(\varphi_{\mathcal T})\geq\gamma J^\NA(\varphi_{\mathcal T})$ for every normal ample test configuration $\mathcal T$. 
\end{definition}

A natural extension of Definition \ref{def: Kbeta stability} is the following semistability notion.

\begin{definition}\label{def:Kbeta-semistability}
Fix $\beta\in\mathbb Q_{>1}$. We say that $(X,L)$ is \emph{$\mathrm K^\beta$-semistable} if $\mathrm K^\beta(\varphi_{\mathcal T})\geq0$ for every normal ample test configuration $\mathcal T$.
\end{definition}

This definition did not appear originally in \cite{DZ26}, although it is a natural extension. We will see in Section \ref{sec:k-ss} that its relation with ordinary K-semistability requires an asymptotic refinement.

\begin{proposition}[{\cite[Theorem 9.5, Corollary 9.6]{DZ26}}]\label{prop:Kbeta-dictionary}
Let $1<\beta_1\leq\beta_2$ be rational. Every normal ample test configuration $\mathcal T$ satisfies
\[
 \mathrm K^{\beta_1}(\varphi_{\mathcal T}) \leq\mathrm K^{\beta_2}(\varphi_{\mathcal T}) \leq\mathrm K^\NA(\varphi_{\mathcal T}) \leq\operatorname{DF}(\mathcal T),
\]
and $\mathrm K^\beta(\varphi_{\mathcal T})\nearrow\mathrm K^\NA(\varphi_{\mathcal T})$ as $\beta\to\infty$. Consequently, uniform $\mathrm K^\beta$-stability implies ordinary uniform K-stability.
\end{proposition}

\begin{proof}
The monotonicity and convergence are \cite[Theorem 9.5]{DZ26}, while the second inequality follows from \cite[Corollary 9.6]{DZ26}, since $\mathrm K^\NA$ and $\mathrm K^\beta$ contain the same polarisation derivative. The final inequality is \cite[Proposition 7.15]{BHJ17}.

If $(X,L)$ is uniformly $\mathrm K^\beta$-stable, then there exists $\gamma>0$ such that every normal ample test configuration satisfies
\[
 \mathrm K^\NA(\varphi_{\mathcal T}) \geq\mathrm K^\beta(\varphi_{\mathcal T}) \geq\gamma J^\NA(\varphi_{\mathcal T}).
\]
Thus $(X,L)$ is uniformly K-stable by \cite[Proposition 8.2]{BHJ17}.
\end{proof}

\begin{corollary}\label{prop:finite-beta-one-way}
Let $\beta_0>1$ be rational. If $(X,L)$ is $\mathrm K^{\beta_0}$-semistable, then it is K-semistable. Moreover, it is $\mathrm K^\beta$-semistable for every rational $\beta\geq\beta_0$.
\end{corollary}

\begin{proof}
By \cite[Theorem 9.5]{DZ26}, every normal ample test configuration $\mathcal T$ satisfies
\[
 \mathrm K^{\beta_0}(\varphi_{\mathcal{T}}) \leq\mathrm K^{\beta}(\varphi_{\mathcal{T}}) \leq\mathrm K^\NA(\varphi_{\mathcal{T}})
\]
for all rational $\beta\geq\beta_0$. If the first term is nonnegative for every $\mathcal T$, then so are the other two terms. The middle term proves $\mathrm K^\beta$-semistability for every $\beta\geq\beta_0$, while the last term gives ordinary K-semistability by \cite[Proposition 8.2]{BHJ17}.
\end{proof}

\section{Comparison of the Mabuchi functionals}\label{sec: mixed holder}

We will use intersection-theoretic methods to prove Theorem \ref{thm:algebraic-polarisation-error-intro} and compare $\mathrm K^\NA$ with $\mathrm K^\beta$. We retain the notation $V=(L^n)$ from the preceding section.

\subsection{Passage to a smooth dominant model}

\begin{lemma}\label{lem:smooth-domination}
Let $\mathcal T=(\cX,\cL)$ be a normal ample test configuration for $(X,L)$, compactified over $\mathbb P^1$. There exist a smooth $\mathbb C^*$-equivariant variety $\cY$, equivariant projective birational morphisms
\[
 \mu:\cY\longrightarrow\cX, \qquad \rho:\cY\longrightarrow X\times\mathbb P^1,
\]
and $\cM:=\mu^*\cL$ such that $\cY_{0,\mathrm{red}}$ has simple normal crossings and $(\cY,\cM)$ is a dominant semiample test configuration. It represents the same non-Archimedean potential $\varphi_{\mathcal T}$. Consequently
\[
 \mathrm K^\beta(\varphi_{\mathcal T}),\quad \Ent_{\mathcal T}^\beta,\quad
 J^\NA(\varphi_{\mathcal T}),\quad E^\NA(\varphi_{\mathcal T}),
 \quad\text{and}\quad \|\varphi_{\mathcal T}\|_1
\]
are unchanged by this replacement.
\end{lemma}

This is the standard equivariant-resolution reduction; see \cite[\S2, under ``Compactification and domination'']{DZ26}. The invariance of the non-Archimedean metric under equivariant pullback is built into \cite[Definition 6.1]{BHJ17}.

\subsection{The log-discrepancy obstacle, its psh envelope, and independence of the resolution} \label{sec:model-independence}

We take a normal test configuration $\mathcal T=(\cX,\cL)$ with associated metric $\varphi:=\varphi_{\mathcal T}$ and, using Lemma \ref{lem:smooth-domination}, fix a smooth dominant model $(\cY,\cM)$. For $\beta\in\mathbb Q_{>1}$, we set $\cM_\beta:=\cM+\frac1\beta K^{\log}_{\cY/X\times\mathbb P^1}$ and $\mathcal T_\beta:=(\cY,\cM_\beta)$.

The SNC divisor $(\cY_0)_{\mathrm{red}}$ has a dual complex $\Delta_\cY$. Its vertices correspond to irreducible components of the central fibre, while a face records a nonempty intersection of components; points of a face encode the corresponding monomial valuations. Thus, there is a canonical retraction
\[
 r_\cY:X^{\mathrm{an}}\longrightarrow\Delta_\cY
\]
that sends a valuation to the monomial valuation having the same values on local equations of those components. It fixes $\Delta_\cY$ pointwise. In particular, $A_X\circ r_\cY$ is a continuous piecewise-affine model function whose vertex values are the log discrepancies of the components. For more details see \cite[Appendix A]{BJ23} and \cite[Remark 9.2]{DZ26}.

Although $\varphi$ is psh, the discrepancy term need not preserve psh semipositivity, so $f_{\beta,\cY}$ is not necessarily a psh metric. Instead we make the following definition.

\begin{definition}\label{def:log-discrepancy-obstacle}
For a chosen smooth dominant model $(\cY,\cM)$, we define the \emph{log-discrepancy obstacle}
\[
 f_{\beta,\cY}:=\varphi+\frac1\beta A_X\circ r_\cY
\]
and set
\[
 \varphi_{\beta,\cY}:=P(f_{\beta,\cY})
 =P\left(\varphi+\frac1\beta A_X\circ r_\cY\right).
\]
\end{definition}

By \cite[Lemma 9.1, Remark 9.2 and (74)]{DZ26}, $\varphi_{\beta,\cY}$ is precisely the non-Archimedean psh metric associated to the big model
\[
 \mathcal T_\beta =\left(\cY,\cM+\frac1\beta K^{\log}_{\cY/X\times\mathbb P^1}\right).
\]

\begin{lemma}\label{lem:model-independence}
Let $\cY_1$ and $\cY_2$ be two smooth dominant models of $\mathcal T$. Then $\varphi_{\beta,\cY_1}=\varphi_{\beta,\cY_2}$. We therefore denote their common metric by $\varphi_\beta$.

Moreover, for every smooth dominant model $\cY$,
$\varphi\leq\varphi_\beta\leq f_{\beta,\cY}$ and $0\leq\varphi_\beta-\varphi\leq\frac1\beta A_X\circ r_\cY$.
\end{lemma}

\begin{proof}
Suppose first that $\nu:\cY'\to\cY$ is a higher smooth model. By \cite[Lemma 9.1]{DZ26}, the big models $\left(\cY,\cM+\frac1\beta K^{\log}_{\cY/X\times\mathbb P^1}\right)$ and $\left(\cY',\nu^*\cM+\frac1\beta K^{\log}_{\cY'/X\times\mathbb P^1}\right)$ induce the same non-Archimedean potential. Hence
\[
 \varphi_{\beta,\cY'}=\varphi_{\beta,\cY}.
\]
For two arbitrary smooth dominant models, we take a common smooth equivariant resolution and apply the preceding observation to its two morphisms. This proves model independence.

Since $X$ is smooth, $A_X\geq0$, and hence $\varphi\leq f_{\beta,\cY}$. As $\varphi$ is psh, it is an admissible obstacle in the envelope defining $P(f_{\beta,\cY})$, so $\varphi\leq\varphi_\beta$. By definition, the supremum of the psh obstacles lies below $f_{\beta,\cY}$. Since $f_{\beta,\cY}$ is continuous, the same remains true after upper-semicontinuous regularisation, while the envelope property ensures that the resulting function is psh. Thus
\[
 \varphi_\beta=P(f_{\beta,\cY})\leq f_{\beta,\cY};
\]
see \cite[\S4.4, Lemma 4.4]{BJ25}. Subtracting $\varphi$ proves the remaining inequalities.
\end{proof}

Notice that, although the obstacle $f_{\beta,\cY}$ and the retraction $r_\cY$ depend on $\cY$, their psh envelope does not. When no model is displayed, we will use the notation $P(\varphi+\beta^{-1}A_X)$ for this common envelope.

\subsection{Ordered intersection estimates}\label{sec: intersection estimates}

In what follows, a \emph{semipositive PL model metric} $\psi$ means a metric represented, on a normal compactified model $\rho:\cY\to X\times\mathbb P^1$, by $\cL_\psi=\rho^*p_1^*L+D_\psi$, where $D_\psi$ is vertical and $\cL_\psi$ is relatively nef.

Let $\cX_0=\sum_Eb_EE$ be an irreducible decomposition, with $v_E:=b_E^{-1}\ord_E|_{K(X)}$ the associated valuation. For every component $E$ of the central fibre $\ord_E(D)=b_EA_X(v_E)$, and, equivalently, the model function determined by $D$ has value $A_X(v_E)$ at the vertex corresponding to $E$ (cf. \cite[Remark 9.2]{DZ26}). Since $X$ is smooth, $A_X(v_E)\geq0$.

The canonical retraction fixes the trivial valuation and $A_X(v_{\mathrm{triv}})=0$, so $f_{\beta,\cX}(v_{\mathrm{triv}})=\varphi(v_{\mathrm{triv}})$. By Lemma \ref{lem:model-independence}, $\varphi\leq\varphi_\beta\leq f_{\beta,\cX}$, and hence $\varphi_\beta(v_{\mathrm{triv}})=\varphi(v_{\mathrm{triv}})$. By \cite[Proposition 4.12(ii)]{BJ22}, $\sup\varphi=\varphi(v_{\mathrm{triv}})$ and $\sup\varphi_\beta=\varphi_\beta(v_{\mathrm{triv}})$. Thus the two suprema coincide, and after one common translation we may assume $\sup\varphi=\sup\varphi_\beta=0$. We will call the metrics with $\sup\varphi=\sup\varphi_\beta=0$ \emph{normalised}.

All positive products below are taken in the translation-normalised sense of Section \ref{subsec:DR-positive-intersection}; see also \cite[(70)]{DZ26} and \cite[Definition 2.5]{DR26}. For model metrics $\eta_i$ and numerical classes $\theta_i$ on $X$, we write $\theta_{i,\cY}:=\rho^*p_1^*\theta_i$. If $D_{\eta_i}$ is the vertical divisor representing $\eta_i$ on $\cY$, then the energy pairing is
\[
 (\theta_0,\eta_0)\cdots(\theta_n,\eta_n) := (\theta_{0,\cY}+D_{\eta_0})\cdots (\theta_{n,\cY}+D_{\eta_n}),
\]
where the right-hand side is the ordinary intersection number on $\cY$; see \cite[Definition 3.12]{BJ22} and \cite[Definition 6.11]{BHJ17}.

Here $(\Theta,0)$ denotes the pair consisting of a numerical class $\Theta$ on $X$ and the trivial model function. On a common model $\rho:\cY\to X\times\mathbb P^1$, it is represented by $\rho^*p_1^*\Theta$; thus, for example,
\[
 (\Theta,0)\cdot\cA^n =(\rho^*p_1^*\Theta)\cdot\cA^n
\]
is an ordinary intersection number on $\cY$. Correspondingly, $(\Theta,0)\cdot\langle\cL_\psi^n\rangle$ denotes the mixed positive intersection obtained by taking the supremum, or equivalently the Fujita limit, over nef model approximants dominated by $\psi$.

If $F$ is vertical and $\mathcal A_2,\ldots,\mathcal A_n$ are relatively nef, then
\[
 F^2\cdot\mathcal A_2\cdots\mathcal A_n\leq0
\]
by \cite[Lemma 6.14]{BHJ17} (see also \cite[Lemma 1]{LX14}). If $F$ and $G$ are vertical and $\cA_2,\ldots,\cA_n$ are relatively nef, we define
\[
 B(F,G):=-F\cdot G\cdot\cA_2\cdots\cA_n.
\]
By the above discussion, $B(H,H)\geq0$ for every vertical divisor $H$. Applying this to
$H=F+tG$, for $t\in\mathbb R$, shows that
\[
 B(F,F)+2tB(F,G)+t^2B(G,G)\geq0
\]
for every $t$. The discriminant is therefore nonpositive, and hence
\[
 |F\cdot G\cdot\cA_2\cdots\cA_n|^2 \leq \bigl(-F^2\cdot\cA_2\cdots\cA_n\bigr) \bigl(-G^2\cdot\cA_2\cdots\cA_n\bigr);
\]
see \cite[Definition 3.28, and equation (3.17)]{BJ22}. Also \cite[\S3.3]{BJ22}, and \cite[Lemma 3.24 and Corollary 3.34]{BJ22}, which is expressed in the language of psh pairs.

\begin{lemma}\label{lem:algebraic-dirichlet-bound}
Let $\varphi\leq\eta_i\leq0$, $2\leq i\leq n:=\dim(X)$, be normalised semipositive model metrics. If $G$ is the vertical divisor representing $-\varphi$, then there exists $C_{\mathrm H}>0$ such that
\[
 -G^2\cdot\cL_{\eta_2}\cdots\cL_{\eta_n} \leq C_{\mathrm H}VJ^\NA(\varphi),
\]
where $C_{\mathrm H}$ depends only on $n$.
\end{lemma}

\begin{proof}
Suppose first that $n=1$. In this case there are no metrics $\eta_i$, and the desired estimate is
\[
 -G^2\leq C_{\mathrm H}VJ^\NA(\varphi).
\]
Let $\cL_0:=\rho^*p_1^*L$ and $\cL_\varphi=\cL_0-G$, and put $a_0:=G\cdot\cL_\varphi$, $a_1:=G\cdot\cL_0$.
By \cite[Equation (7.15) and Theorem 7.18]{BJ22} 
\[
 \begin{aligned}
 a_0 &=G\cdot\cL_\varphi =\int_{X^{\mathrm{an}}}(-\varphi)\,\operatorname{MA}(\varphi),\\
 a_1 &=G\cdot\cL_0 =\int_{X^{\mathrm{an}}} (-\varphi)\,\operatorname{MA}(0).
 \end{aligned}
\]
Both measures are positive because $\varphi$ and $0$ are $c_1(L)$-psh, while $-\varphi\geq0$. Consequently, $a_0,a_1\geq0$. The relative Hodge index theorem gives $a_1-a_0 =G\cdot(\cL_0-\cL_\varphi)=G^2\leq0$. Furthermore,
\[
 \begin{aligned}
 a_0+a_1
 &=G\cdot(\cL_\varphi+\cL_0)\\
 &=\cL_0^2-\cL_\varphi^2\\
 &=2V\bigl(E^\NA(0)-E^\NA(\varphi)\bigr)
 =2VJ^\NA(\varphi),
 \end{aligned}
\]
where we used $\sup\varphi=0$. Hence, since $a_0,a_1\geq0$, $a_0-a_1\leq a_0\leq a_0+a_1$ and thus
\[
 -G^2=a_0-a_1\leq a_0+a_1 =2VJ^\NA(\varphi).
\]
This proves the assertion when $n=1$. We thus assume $n\geq2$.

For two semipositive model metrics $u,v$, represented on a common model, we define
\[
 \mathfrak d(u,v) :=\max_{0\leq j\leq n-1} \bigl[-(D_u-D_v)^2\cdot\cL_u^j \cdot\cL_v^{n-1-j}\bigr],
\]
and $\mathfrak d(u):=\mathfrak d(u,0)$. Suppose first that $u\leq v$ and set $F:=D_v-D_u$. 

By \cite[Equation (7.15) and Theorem 7.18]{BJ22},
\[
 a_j =F\cdot\cL_v^j\cdot\cL_u^{n-j} =\int_{X^{\mathrm{an}}} (v-u)\,\operatorname{MA}\bigl(v^{\langle j\rangle},u^{\langle n-j\rangle}\bigr).
\]
The mixed measure in this integral is positive, since $u$ and $v$ are $c_1(L)$-psh, and the integrand is nonnegative because $u\leq v$. Hence $a_j\geq0$. Furthermore,
\[
 \begin{aligned}
 a_{j+1}-a_j
 &=F\cdot\cL_v^j\cdot\cL_u^{n-1-j}
 \cdot(\cL_v-\cL_u)\\
 &=F^2\cdot\cL_v^j\cdot\cL_u^{n-1-j} \leq 0,
 \end{aligned}
\]
where the inequality follows from the relative Hodge index theorem \cite[Lemma 6.14]{BHJ17}. Hence $a_{j+1}\leq a_j$. 

Since $F=\cL_v-\cL_u$, by \cite[Lemma 3.26]{BJ22} we have
\[
 \begin{aligned}
 \sum_{j=0}^na_j &=F\cdot\sum_{j=0}^n \cL_v^j\cdot\cL_u^{n-j}\\
 &=(\cL_v-\cL_u)\cdot \sum_{j=0}^n\cL_v^j\cdot\cL_u^{n-j}\\
 &=\cL_v^{n+1}-\cL_u^{n+1}.
 \end{aligned}
\]
For $0\leq j\leq n-1$, we have $a_j-a_{j+1} = -F^2\cdot\cL_v^j\cdot\cL_u^{n-1-j}$ and since every $a_j$ is nonnegative, we obtain
\begin{equation}\label{eq: lemma 34 calc}
 \begin{aligned}
 -F^2\cdot\cL_v^j\cdot\cL_u^{n-1-j} &=a_j-a_{j+1}\\
 &\leq a_j\\
 &\leq\sum_{k=0}^na_k\\
 &=(n+1)V\bigl(E^\NA(v)-E^\NA(u)\bigr).
 \end{aligned}
 \end{equation}
Taking the maximum over $0\leq j\leq n-1$ gives
\[
 \mathfrak d(u,v) \leq(n+1)V\bigl(E^\NA(v)-E^\NA(u)\bigr).
\]

We apply this estimate first to $u=\varphi$, $v=0$, and then to $u=\eta_i$, $v=0$. The common normalisation and $\varphi\leq\eta_i\leq0$ give
\[
 \mathfrak d(\varphi) \leq(n+1)VJ^\NA(\varphi), \qquad \mathfrak d(\eta_i) \leq(n+1)VJ^\NA(\eta_i) \leq(n+1)VJ^\NA(\varphi).
\]

We now set $\alpha_n:=2^{1-n}$, $\zeta_0:=0$, $\zeta_1:=\varphi$, and $\zeta_i:=\eta_i$ for $2\leq i\leq n$. In the notation of
\cite[Definition 3.28]{BJ22},
\[
 \|\zeta_0-\zeta_1\|^2_{(\theta,\zeta_2)\cdots(\theta,\zeta_n)} =-(0,\zeta_0-\zeta_1)^2 \cdot(\theta,\zeta_2)\cdots(\theta,\zeta_n),
\]
where $\theta=c_1(L)$. In our notation, $\zeta_0-\zeta_1=-\varphi$ is represented by $G$, while $(\theta,\eta_i)$ is represented by $\cL_{\eta_i}$. Consequently,
\[
 \|\zeta_0-\zeta_1\|^2_{(\theta,\zeta_2)\cdots(\theta,\zeta_n)} =-G^2\cdot\cL_{\eta_2}\cdots\cL_{\eta_n}.
\]

Moreover, the quantity $d_\theta$ of \cite[Definition 3.29]{BJ22} agrees in our notation with $\mathfrak d$. Thus \cite[Corollary 3.34]{BJ22} gives
\[
 \begin{aligned}
 -G^2\cdot\cL_{\eta_2}\cdots\cL_{\eta_n}
 &\leq C_n\mathfrak d(0,\varphi)^{\alpha_n}
 \max_{0\leq i\leq n}\mathfrak d(\zeta_i)^{1-\alpha_n}.
 \end{aligned}
\]
The definition of $\mathfrak d$ is symmetric in its two arguments, since interchanging them merely replaces $j$ by $n-1-j$. Hence $\mathfrak d(0,\varphi)=\mathfrak d(\varphi)$. Every quantity on the right-hand side is therefore bounded by $B:=(n+1)VJ^\NA(\varphi)$.
It follows that
\[
 -G^2\cdot\cL_{\eta_2}\cdots\cL_{\eta_n}
 \leq C_nB^{\alpha_n}B^{1-\alpha_n}
 =C_n(n+1)VJ^\NA(\varphi).
\]
This proves the lemma, with $C_{\mathrm H}=C_n(n+1)$.
\end{proof}

For comparison, \cite[Lemma 2.8]{BJ23} proves a general dimension-dependent H\"older estimate for the polarisation derivative. The ordered square-root estimate below is sharper in the present model setting, and we prove it directly by intersection theory.

\begin{theorem}\label{lem:algebraic-polarisation-error}
For every $\theta\in N^1(X)_{\mathbb R}$ there is a constant
$C_\theta>0$, depending only on $(X,L)$ and $\theta$, such that
\begin{equation}\label{eq:algebraic-polarisation-error}
 |\nabla_\theta E^\NA(\psi)-\nabla_\theta E^\NA(\varphi)|
 \leq C_\theta
 \bigl(E^\NA(\psi)-E^\NA(\varphi)\bigr)^{1/2}
 J^\NA(\varphi)^{1/2}
\end{equation}
whenever $\varphi$ is a semipositive PL metric,
$\varphi\leq\psi$, and $\sup\varphi=\sup\psi=0$. Here $\psi$ may be either a semipositive PL model metric or the finite-energy psh metric associated with a rational dominant big model.
\end{theorem}

\begin{proof}
Suppose first that both metrics are semipositive model metrics. We choose a common compactified model $\rho:\cY\to X\times\mathbb P^1$ and let $F:=D_\psi-D_\varphi$, $G:=\cL_0-\cL_\varphi$ and $\cL_0:=\rho^*p_1^*L$. Thus $F$ represents the nonnegative model function $\psi-\varphi$, while $G$ represents $-\varphi$.

Since $\sup\varphi=\sup\psi=0$, we have
\[
 E^\NA(\psi)-E^\NA(\varphi) =J^\NA(\varphi)-J^\NA(\psi).
\]
Since $\varphi\leq\psi$, monotonicity of $E^\NA$ gives $E^\NA(\psi)-E^\NA(\varphi)\geq0$. Furthermore, the common normalisation implies $J^\NA(\varphi)=-E^\NA(\varphi)$ and $J^\NA(\psi)=-E^\NA(\psi)$. Since $J^\NA(\psi)\geq0$, we obtain
\[
 E^\NA(\psi)-E^\NA(\varphi) =J^\NA(\varphi)-J^\NA(\psi) \leq J^\NA(\varphi).
\]

For every $0\leq j\leq n-1$, \eqref{eq: lemma 34 calc} gives
\[
 -F^2\cdot\cL_\psi^j\cdot\cL_\varphi^{n-1-j} \leq (n+1)V\bigl(E^\NA(\psi)-E^\NA(\varphi)\bigr).
\]
Since every factor $\cL_\psi$ or $\cL_\varphi$ lies between $\cL_\varphi$ and $\cL_0$, Lemma \ref{lem:algebraic-dirichlet-bound} also gives
\[
 -G^2\cdot\cL_\psi^j\cdot\cL_\varphi^{n-1-j}
 \leq C_{\mathrm H}VJ^\NA(\varphi).
\]
Hence, using the Cauchy-Schwarz inequality in the preceding discussion, we have
\[
\left| F\cdot G\cdot \cL_\psi^j\cdot\cL_\varphi^{n-1-j}
 \right| \leq CV \bigl(E^\NA(\psi)-E^\NA(\varphi)\bigr)^{1/2} J^\NA(\varphi)^{1/2},
\]
where $C$ depends only on $n$.

We next estimate the intersections with $\theta$. We choose ample real classes $\Theta_+$ and $\Theta_-$ such that $\theta=\Theta_+-\Theta_-$ and choose $m>0$ sufficiently large such that $mL-\Theta_+$ and $mL-\Theta_-$ are nef. We denote their horizontal pullbacks to $\cY$ by $\Theta_{+,\cY}$ and $\Theta_{-,\cY}$.

Let $\operatorname{MA}_{\Theta_\pm}\bigl(\psi^{\langle j\rangle},\varphi^{\langle n-1-j\rangle}\bigr)$ denote the positive mixed model Monge--Amp\`ere measure determined by $\Theta_\pm$, $j$ copies of $\psi$, and $n-1-j$ copies of $\varphi$. By \cite[Equation (7.15) and Theorem 7.18]{BJ22}, the intersection--integration identity gives
\[
 F\cdot\Theta_{\pm,\cY}\cdot \cL_\psi^j\cdot\cL_\varphi^{n-1-j}
 =\int_{X^{\mathrm{an}}}(\psi-\varphi)\,\operatorname{MA}_{\Theta_\pm}\bigl(\psi^{\langle j\rangle},\varphi^{\langle n-1-j\rangle}\bigr)\geq0,
\]
and, since $mL-\Theta_\pm$ is nef,
\[
 F\cdot \bigl(m\cL_0-\Theta_{\pm,\cY}\bigr)\cdot \cL_\psi^j\cdot\cL_\varphi^{n-1-j} \geq0.
\]
Hence
\[
 0\leq F\cdot\Theta_{\pm,\cY}\cdot \cL_\psi^j\cdot\cL_\varphi^{n-1-j} \leq mF\cdot\cL_0\cdot \cL_\psi^j\cdot\cL_\varphi^{n-1-j}.
\]

Since $\cL_0=\cL_\varphi+G$,
\[
F\cdot\cL_0\cdot \cL_\psi^j\cdot\cL_\varphi^{n-1-j}= F\cdot\cL_\psi^j\cdot\cL_\varphi^{n-j} + F\cdot G\cdot \cL_\psi^j\cdot\cL_\varphi^{n-1-j}.
\]
The first term equals $a_j$ in the notation of the proof of Lemma \ref{lem:algebraic-dirichlet-bound}, thus $a_j\geq 0$ and we thus have
\[
 0\leq a_j = F\cdot\cL_\psi^j\cdot\cL_\varphi^{n-j}
 \leq \sum_{k=0}^na_k= \sum_{k=0}^n F\cdot\cL_\psi^k\cdot\cL_\varphi^{n-k} =(n+1)V \bigl(E^\NA(\psi)-E^\NA(\varphi)\bigr).
\]
Since $\theta_{\cY}=\Theta_{+,\cY}-\Theta_{-,\cY}$ and both intersections with $\Theta_{+,\cY}$ and $\Theta_{-,\cY}$ are nonnegative, the preceding estimates give
\[
 \begin{aligned}
 \left|
 F\cdot\theta_{\cY}\cdot
 \cL_\psi^j\cdot\cL_\varphi^{n-1-j}
 \right| &\leq
 F\cdot\Theta_{+,\cY}\cdot
 \cL_\psi^j\cdot\cL_\varphi^{n-1-j} +F\cdot\Theta_{-,\cY}\cdot
 \cL_\psi^j\cdot\cL_\varphi^{n-1-j}\\
 &\leq 2mF\cdot\cL_0\cdot
 \cL_\psi^j\cdot\cL_\varphi^{n-1-j}\\
 &\leq 2m(n+1)V\bigl(E^\NA(\psi)-E^\NA(\varphi)\bigr) +2mCV \bigl(E^\NA(\psi)-E^\NA(\varphi)\bigr)^{1/2}
 J^\NA(\varphi)^{1/2}.
 \end{aligned}
\]

We next absorb the first term into the second. Since $E^\NA(\psi)-E^\NA(\varphi) =J^\NA(\varphi)-J^\NA(\psi) \leq J^\NA(\varphi)$, and by the monotonicity of $f(t)=t^{1/2}$ we obtain
\[
 \left|F\cdot\theta_{\cY}\cdot \cL_\psi^j\cdot\cL_\varphi^{n-1-j} \right| \leq C_\theta V\bigl(E^\NA(\psi)-E^\NA(\varphi)\bigr)^{1/2} J^\NA(\varphi)^{1/2}.
\]

Since
\[
 \theta_{\cY}\cdot \bigl(\cL_\psi^n-\cL_\varphi^n\bigr) = \sum_{j=0}^{n-1} F\cdot\theta_{\cY}\cdot \cL_\psi^j\cdot\cL_\varphi^{n-1-j},
\]
we obtain
\[
 \left| \theta_{\cY}\cdot \bigl(\cL_\psi^n-\cL_\varphi^n\bigr) \right| \leq C_\theta V \bigl(E^\NA(\psi)-E^\NA(\varphi)\bigr)^{1/2}
 J^\NA(\varphi)^{1/2}.
\]
Now, by \cite[Definition 2.1 and Lemma 2.4]{BJ23}, we obtain
\[
\begin{aligned}
 |\nabla_\theta E^\NA(\psi)
 -\nabla_\theta E^\NA(\varphi)|&\leq |\operatorname{tr}_L(\theta)| \bigl(E^\NA(\psi)-E^\NA(\varphi)\bigr) + \frac1V \left| \theta_{\cY}\cdot \bigl(\cL_\psi^n-\cL_\varphi^n\bigr) \right|\\
 &\leq C_\theta \bigl(E^\NA(\psi)-E^\NA(\varphi)\bigr)^{1/2} J^\NA(\varphi)^{1/2}.
\end{aligned}
\]
This gives the result when both metrics are semipositive normalised metrics.

We now let $\psi$ be the finite-energy psh metric associated with a rational dominant big model. After one common central twist,
Fujita approximation (see for instance \cite[(2.7)]{DR26}) gives semipositive PL metrics $\psi_m\leq\psi$, with nef model classes $\cA_m:=\cL_{\psi_m}$ such that $\cA_m^{n+1}\longrightarrow \langle\cL_\psi^{n+1}\rangle$ and, for each sign,
\[
 \Theta_{\pm,\cY_m}\cdot\cA_m^n \longrightarrow \Theta_\pm\cdot\langle\cL_\psi^n\rangle.
\]

We now set $\widetilde\psi_m:=\max\{\varphi,\psi_m\}$ and $\widetilde\cA_m:=\cL_{\widetilde\psi_m}$. By \cite[Proposition 3.6(iii)]{BJ22}, $\widetilde\psi_m$ is a semipositive PL metric. Furthermore, $\varphi\leq\widetilde\psi_m\leq\psi\leq0$ so $\sup\widetilde\psi_m=0$, and $\widetilde\cA_m$ is still an admissible nef approximant from below.

Let $H_m:=D_{\widetilde\psi_m}-D_{\psi_m}$. The model function represented by $H_m$ is $\widetilde\psi_m-\psi_m$, hence is nonnegative. Therefore
\[
 \widetilde\cA_m^{n+1}-\cA_m^{n+1} = \sum_{j=0}^n H_m\cdot\widetilde\cA_m^j\cdot\cA_m^{n-j} \geq0,
\]
because every summand is the integral of this nonnegative model function against a positive mixed Monge--Amp\`ere measure. Similarly,
\[
\Theta_{\pm,\cY_m}\cdot \bigl(\widetilde\cA_m^n-\cA_m^n\bigr) = \sum_{j=0}^{n-1}H_m\cdot\Theta_{\pm,\cY_m}\cdot \widetilde\cA_m^j\cdot\cA_m^{n-1-j} \geq0.
\]

Since $\widetilde\cA_m$ is an admissible nef model class dominated by the big model defining $\psi$, the definition of the positive intersection product \cite[\S2.1.2]{DR26} gives $\widetilde\cA_m^{n+1} \leq\langle\cL_\psi^{n+1}\rangle$ and, in the pseudoeffective order, $\widetilde\cA_m^n \leq\langle\cL_\psi^n\rangle$. Since $\Theta_{\pm,\cY_m}$ are both nef, the latter inequality implies $\Theta_{\pm,\cY_m}\cdot\widetilde\cA_m^n \leq \Theta_\pm\cdot\langle\cL_\psi^n\rangle$. Together with the preceding lower bounds, we therefore have $\cA_m^{n+1} \leq\widetilde\cA_m^{n+1} \leq\langle\cL_\psi^{n+1}\rangle$ and $\Theta_{\pm,\cY_m}\cdot\cA_m^n \leq \Theta_{\pm,\cY_m}\cdot\widetilde\cA_m^n \leq \Theta_\pm\cdot\langle\cL_\psi^n\rangle$. The squeeze theorem hence gives $\widetilde\cA_m^{n+1} \longrightarrow \langle\cL_\psi^{n+1}\rangle$ and $\Theta_{\pm,\cY_m}\cdot\widetilde\cA_m^n \longrightarrow \Theta_\pm\cdot\langle\cL_\psi^n\rangle$. Since $\theta=\Theta_+-\Theta_-$, subtracting the two limits yields
\[
 \theta_{\cY_m}\cdot\widetilde\cA_m^n \longrightarrow \theta\cdot\langle\cL_\psi^n\rangle.
\]
The positive-intersection formula for the energy and \cite[Definition 2.1 and Lemma 2.4]{BJ23} now imply $E^\NA(\widetilde\psi_m)\longrightarrow E^\NA(\psi)$ and $\nabla_\theta E^\NA(\widetilde\psi_m)
 \longrightarrow\nabla_\theta E^\NA(\psi)$.

Applying the already proved model case to $(\varphi,\widetilde\psi_m)$ gives
\[
|\nabla_\theta E^\NA(\widetilde\psi_m)
 -\nabla_\theta E^\NA(\varphi)|\leq C_\theta \bigl(E^\NA(\widetilde\psi_m)-E^\NA(\varphi)\bigr)^{1/2}
 J^\NA(\varphi)^{1/2}.
\]
Hence, passing to the limit proves the assertion for $\psi$.
\end{proof}

\begin{corollary}\label{cor:algebraic-polarisation-linear}
For every $\theta\in N^1(X)_{\mathbb R}$ there is a constant $C_\theta'>0$ such that every semipositive PL model metric $\varphi$, normalised by $\sup\varphi=0$, satisfies
\begin{equation}\label{eq:algebraic-polarisation-linear}
 |\nabla_\theta E^\NA(\varphi)| \leq C_\theta'J^\NA(\varphi).
\end{equation}
\end{corollary}

\begin{proof}
We apply Theorem \ref{lem:algebraic-polarisation-error} with $\psi=0$. Since $\nabla_\theta E^\NA(0)=0$ and $E^\NA(0)-E^\NA(\varphi)=J^\NA(\varphi)$, the result follows.
\end{proof}

\subsection{Comparison of Mabuchi functionals}\label{sec: comparison of stability}\label{sec: mabuchi comparison}

We retain the notation of the preceding subsections. Thus $\mathcal T$ is a normal ample test configuration with associated semipositive PL metric $\varphi=\varphi_{\mathcal T}$, and $\varphi_\beta$ is its log-discrepancy metric.

\begin{lemma}\label{lem:distance-entropy}
We have
\begin{equation}\label{eq:distance-entropy}
 d_1^\NA(\varphi,\varphi_\beta)
 =E^\NA(\varphi_\beta)-E^\NA(\varphi)
 =\frac1\beta\Ent_{\mathcal T}^\beta.
\end{equation}
\end{lemma}

\begin{proof}
Since $X$ is smooth, it has the envelope property. Since $\varphi\leq\varphi_\beta$, \cite[Theorem 5.5(i),(iii)]{BJ25} gives
\[
 d_1^\NA(\varphi,\varphi_\beta) =E^\NA(\varphi_\beta)-E^\NA(\varphi).
\]
Together with \cite[Proposition 9.4]{DZ26}, we obtain 
\[
 \Ent_{\mathcal T}^\beta =\beta\bigl(E^\NA(\varphi_\beta)-E^\NA(\varphi)\bigr),
\]
which proves the assertion.
\end{proof}

Also, recall the following result.

\begin{lemma}[{\cite[Corollary 9.6]{DZ26}}]\label{lem:entropy-bridge}
For every normal ample test configuration $\mathcal T$ and every rational $\beta>1$, its log-discrepancy metric satisfies $\Ent_{\mathcal T}^\beta \geq\Ent^\NA(\varphi_\beta)$.
\end{lemma}

We first record the approximation step that passes an inequality from ample test configurations to their associated big models appearing above.

\begin{proposition}\label{prop:input}
Fix a normal ample test configuration $\mathcal T$ and a rational number $\beta>1$. Let $\varphi_\beta$ be the corresponding log-discrepancy metric. If there exists $\Lambda\in\mathbb R$ such that
\[
 \mathrm K^\NA(\psi)\geq\Lambda\|\psi\|_1
\]
for every normal ample test-configuration metric $\psi$, then
\[
 \mathrm K^\NA(\varphi_\beta)
 \geq\Lambda\|\varphi_\beta\|_1.
\]
\end{proposition}

\begin{proof}
By \cite[Corollary A and its proof]{Tru26}, together with \cite[Lemma 4.5]{Li21}, there exist smooth ample test-configuration metrics $\psi_m$ such that $d_1^\NA(\psi_m,\varphi_\beta)\longrightarrow0$ and $\Ent^\NA(\psi_m)\longrightarrow\Ent^\NA(\varphi_\beta)$. The strong continuity of $\nabla_{K_X}E^\NA$ therefore gives $\mathrm K^\NA(\psi_m)\longrightarrow\mathrm K^\NA(\varphi_\beta)$, while Lemma \ref{lem:distance-to-fixed-subset} gives
\[
 \bigl|\|\psi_m\|_1-\|\varphi_\beta\|_1\bigr|
 \leq d_1^\NA(\psi_m,\varphi_\beta)\longrightarrow0.
\]
Passing to the limit in
$\mathrm K^\NA(\psi_m)\geq\Lambda\|\psi_m\|_1$
proves the assertion.
\end{proof}

\begin{theorem}\label{thm:mabuchi-comparison}
Let $\mathcal T$ be a normal ample test configuration with associated metric $\varphi$, and let $\varphi_\beta$ be its log-discrepancy metric. There is a constant $C>0$, depending only on $(X,L)$, such that
\[
 \mathrm K^\beta(\varphi_{\mathcal T})
 \geq\mathrm K^\NA(\varphi_\beta)
 -C\,d_1^\NA(\varphi,\varphi_\beta)^{1/2}J^\NA(\varphi)^{1/2}.
\]
\end{theorem}

\begin{proof}
After one common translation, we may assume that $\sup\varphi=\sup\varphi_\beta=0$. By Lemma \ref{lem:entropy-bridge} and the definitions of the two Mabuchi functionals,
\[
\begin{aligned}
 \mathrm K^\beta(\varphi_{\mathcal T})
 &=\Ent_{\mathcal T}^\beta-\nabla_{-K_X}E^\NA(\varphi)\\
 &\geq\Ent^\NA(\varphi_\beta)-\nabla_{-K_X}E^\NA(\varphi)\\
 &=\mathrm K^\NA(\varphi_\beta)
 +\nabla_{-K_X}E^\NA(\varphi_\beta)
 -\nabla_{-K_X}E^\NA(\varphi).
\end{aligned}
\]
The result now follows from Theorem \ref{lem:algebraic-polarisation-error}, applied with $\theta=c_1(-K_X)$, and Lemma \ref{lem:distance-entropy}.
\end{proof}

\begin{corollary}\label{cor:transfer}
Suppose that there exists $\Lambda\geq0$ such that
\[
 \mathrm K^\NA(\varphi_\beta)
 \geq\Lambda\|\varphi_\beta\|_1
\]
for the log-discrepancy metric of every normal ample test configuration and every rational $\beta>1$. Then there exist $\varepsilon_\beta\geq0$, depending only on $(X,L)$, $\Lambda$, and $\beta$, such that $\varepsilon_\beta\to0$ and
\[
 \mathrm K^\beta(\varphi_{\mathcal T}) \geq\bigl(\Lambda-\varepsilon_\beta\bigr) \|\varphi_{\mathcal T}\|_1
\]
for every normal ample test configuration $\mathcal T$.
\end{corollary}

\begin{proof}
Let $\varphi = \varphi_{\mathcal{T}}$ be the metric associated to the test configuration. If $\|\varphi\|_1=0$, then $\varphi$ is constant by Lemma \ref{lem:zero-norm}. The translation invariance of the algebraic formula for $\mathrm K^\beta$ therefore gives $\mathrm K^\beta(\varphi)=0$, and there is nothing to prove.

Suppose that $\|\varphi\|_1>0$. By Corollary \ref{cor:algebraic-polarisation-linear} and Lemma \ref{lem:norm-comparison}, after enlarging a constant $C_0>0$ depending only on $(X,L)$, we have
\[
 |\nabla_{-K_X}E^\NA(\varphi)|
 \leq C_0\|\varphi\|_1.
\]
If $\mathrm K^\beta(\varphi)\geq\Lambda\|\varphi\|_1$, the required inequality is immediate. We may therefore assume that $\mathrm K^\beta(\varphi)<\Lambda\|\varphi\|_1$. Lemma \ref{lem:distance-entropy} then gives
\[
\begin{aligned}
 \beta d_1^\NA(\varphi,\varphi_\beta)
 &=\Ent_{\mathcal T}^\beta\\
 &=\mathrm K^\beta(\varphi)
 +\nabla_{-K_X}E^\NA(\varphi)\\
 &\leq(\Lambda+C_0)\|\varphi\|_1.
\end{aligned}
\]
Moreover, Lemma \ref{lem:distance-to-fixed-subset} and Lemma \ref{lem:norm-comparison} give inequalities $\|\varphi_\beta\|_1 \geq\|\varphi\|_1-d_1^\NA(\varphi,\varphi_\beta)$ and $J^\NA(\varphi)\leq c_J^{-1}\|\varphi\|_1$. Theorem \ref{thm:mabuchi-comparison} therefore gives
\[
\begin{aligned}
 \mathrm K^\beta(\varphi)
 &\geq\Lambda\|\varphi_\beta\|_1
 -C\,d_1^\NA(\varphi,\varphi_\beta)^{1/2}J^\NA(\varphi)^{1/2}\\
 &\geq\left(\Lambda-
 \frac{\Lambda(\Lambda+C_0)}{\beta}
 -C\left(\frac{\Lambda+C_0}{c_J\beta}\right)^{1/2}\right)
 \|\varphi\|_1.
\end{aligned}
\]
Defining 
\[\varepsilon_\beta: = \frac{\Lambda(\Lambda+C_0)}{\beta}
 +C\left(\frac{\Lambda+C_0}{c_J\beta}\right)^{1/2}
\]
we see immediately that $\varepsilon_\beta \geq 0$ and that $\varepsilon_\beta = O(\beta^{-1/2})$, i.e. $\varepsilon_\beta\to0$. Furthermore, since $C_0$, $c_J$ depend only on $(X,L)$, we conclude that $\varepsilon_\beta$ also depends only on $(X,L)$, $\Lambda$, and $\beta$. This concludes the proof.
\end{proof}

\begin{remark}\label{rem:BJ-alternative}
The general H\"older estimate of \cite[Lemma~2.8]{BJ23} is also strong enough for the conclusion of Corollary \ref{cor:transfer}. Combining it with Lemma \ref{lem:distance-entropy} gives an error of order $O(\beta^{-\alpha_n})$, for a dimension-dependent exponent $\alpha_n\in(0,1)$. The purpose of Theorem \ref{lem:algebraic-polarisation-error} is therefore to obtain a sharper error term and the explicit constant $\varepsilon_\beta$.
\end{remark}

\section{Consequences for stability and semistability} \label{sec:stability-consequences}

We will now use Corollary \ref{cor:transfer} to obtain a direct comparison between $\widehat{\mathrm{K}}$-stability and $\mathrm{K}^\beta$-stability.

\subsection{Uniform stability when the automorphism group is discrete}

\begin{theorem}\label{cor:algebraic}
Let $(X,L)$ be smooth with $\Aut^0(X,L)$ trivial. The following are equivalent:
\begin{enumerate}[label=\textup{(\roman*)}]
\item $(X,L)$ is uniformly $\widehat{\mathrm K}$-polystable;
\item $(X,L)$ is uniformly $\mathrm K^\beta$-stable for every sufficiently large rational $\beta$;
\item $(X,L)$ is uniformly $\mathrm K^\beta$-stable for some rational $\beta>1$.
\end{enumerate}
\end{theorem}

\begin{proof}
Assume first that $(X,L)$ is uniformly $\widehat{\mathrm K}$-polystable. Since $\Aut^0(X,L)$ is trivial, there exists $\sigma>0$ such that
\[
 \mathrm K^\NA(\psi)\geq\sigma\|\psi\|_1
\]
for every $\psi\in\mathcal E^{1,\NA}(L)$. Proposition \ref{prop:input} and Corollary \ref{cor:transfer}, applied with $\Lambda=\sigma$, gives $\varepsilon_\beta=O(\beta^{-1/2})$ such that
\[
 \mathrm K^\beta(\varphi_{\mathcal T})
 \geq(\sigma-\varepsilon_\beta)\|\varphi_{\mathcal T}\|_1.
\]
We thus choose $\beta_0>1$ such that $\varepsilon_\beta\leq\sigma/2$ whenever $\beta\geq\beta_0$. Lemma \ref{lem:norm-comparison} then gives
\[
 \mathrm K^\beta(\varphi_{\mathcal T}) \geq\frac\sigma2\|\varphi_{\mathcal T}\|_1 \geq\frac{\sigma c_J}{2}J^\NA(\varphi_{\mathcal T})
\]
for every normal ample test configuration and every rational $\beta\geq\beta_0$. This proves \textup{(ii)}, and \textup{(ii)}$\Rightarrow$\textup{(iii)} is immediate.

Conversely, suppose that \textup{(iii)} holds. By Proposition \ref{prop:Kbeta-dictionary} $(X,L)$ is uniformly K-stable. Since $\Aut^0(X,L)$ is trivial, \cite[Corollary B]{Tru26} gives uniform $\widehat{\mathrm K}$-polystability, proving \textup{(i)}.
\end{proof}

\subsection{Reduced \texorpdfstring{$\mathrm K^\beta$}{K-beta}-polystability}
\label{sec:reduced-Kbeta}

Let $G:=\operatorname{Aut}^0(X,L)$, and suppose that $G$ is reductive. Fix a maximal algebraic torus $\mathbb T\subset G$, and let $T\subset\mathbb T$ be its maximal compact subtorus. We write
\[
 N_{\mathbb Q}:=\operatorname{Hom}(\mathbb G_m,\mathbb T)\otimes_{\mathbb Z}\mathbb Q,
 \qquad N_{\mathbb R}:=N_{\mathbb Q}\otimes_{\mathbb Q}\mathbb R.
\]
If $\varphi$ is a $\mathbb T$-equivariant metric and $\xi\in N_{\mathbb Q}$, the rational product twist of the corresponding test configuration is denoted by $\mathcal T_\xi$, and its metric by $\xi\star\varphi$. The proof of \cite[Theorem 8.19]{BJ26} shows that this action extends continuously to $N_{\mathbb R}$ and that
\[
 d_1^\NA(\xi\star\varphi,\xi\star\psi)
 =d_1^\NA(\varphi,\psi).
\]
We set
\[
 q_{\mathbb T}(\varphi)
 :=\inf_{\xi\in N_{\mathbb R}}\|\xi\star\varphi\|_1.
\]
It follows from the preceding equality and Lemma \ref{lem:distance-to-fixed-subset} that
\begin{equation}\label{eq:reduced-norm-Lipschitz}
 |q_{\mathbb T}(\varphi)-q_{\mathbb T}(\psi)|
 \leq d_1^\NA(\varphi,\psi).
\end{equation}

For $\xi\in N_{\mathbb Q}$, the value of the Futaki character on $\xi$ is
\[
 \operatorname{Fut}_{X,L}(\xi)
 :=\mathrm K^\NA(\xi\star0).
\]
We write $\operatorname{Fut}_{X,L}=0$ when this character vanishes identically. In that case the non-Archimedean Mabuchi functional is invariant under product twists; \cite[Corollary 8.8]{BJ26}.

For a rational number $\beta>1$, we define
\[
 \mathrm K_{\mathbb T,\mathrm{red}}^\beta(\varphi_{\mathcal T}) :=\sup_{\xi\in N_{\mathbb Q}} \mathrm K^\beta(\varphi_{\mathcal T_\xi}).
\]
\begin{definition}
 We say that $(X,L)$ is \emph{uniformly reduced $\mathrm K^\beta$-polystable} if there exists $\gamma>0$ such that
\[
 \mathrm K_{\mathbb T,\mathrm{red}}^\beta(\varphi_{\mathcal T})
 \geq\gamma q_{\mathbb T}(\varphi_{\mathcal T})
\]
for every $\mathbb T$-equivariant normal ample test configuration $\mathcal T$.
\end{definition}

This stability condition is close to, but not identical with, the reduced stability condition of \cite[\S10, Equations (86), (89) and Theorem 10.4]{DZ26}. In particular, Darvas--Zhang use the reduced radial $J$-functional and a $G$-uniform Mabuchi inequality on log-discrepancy models, while here we optimise the fixed $\beta$-Mabuchi invariant over rational product twists of an ample test configuration.

We assume from now on that $\operatorname{Fut}_{X,L}=0$.

\begin{lemma}\label{lem:reduced-Kbeta-basic}
Every $\mathbb T$-equivariant normal ample test-configuration metric satisfies
\[
 \mathrm K_{\mathbb T,\mathrm{red}}^\beta(\varphi)
 \leq\mathrm K^\NA(\varphi).
\]
If $q_{\mathbb T}(\varphi)=0$, then
\[
 \mathrm K_{\mathbb T,\mathrm{red}}^\beta(\varphi)=0.
\]
\end{lemma}

\begin{proof}
For every $\xi\in N_{\mathbb Q}$, Proposition \ref{prop:Kbeta-dictionary} and the vanishing of the Futaki character give
\[
 \mathrm K^\beta(\xi\star\varphi)
 \leq\mathrm K^\NA(\xi\star\varphi)
 =\mathrm K^\NA(\varphi).
\]
Taking the supremum proves the first assertion.

Suppose that $q_{\mathbb T}(\varphi)=0$. By the description of the real product locus in \cite[Proposition 4.11, Definition 4.12 and Corollary 4.14]{BJ26}, an algebraic metric at zero reduced distance is a rational product metric modulo translation. Hence some rational product twist takes $\varphi$ to a constant metric. For this twist the $\mathrm K^\beta$-invariant vanishes, so $\mathrm K_{\mathbb T,\mathrm{red}}^\beta(\varphi)\geq0$. On the other hand, $\mathrm K^\NA(\varphi)=0$ by the vanishing of the Futaki character, and the first assertion gives the reverse inequality.
\end{proof}

\begin{lemma}\label{lem:optimal-rational-twist}
Let $\varphi$ be a $\mathbb T$-equivariant ample test-configuration metric with $q_{\mathbb T}(\varphi)>0$. Then, for every $a>1$, there exists $\xi\in N_{\mathbb Q}$ such that $q_{\mathbb T}(\xi\star\varphi)=q_{\mathbb T}(\varphi)$ and $\|\xi\star\varphi\|_1<a q_{\mathbb T}(\varphi)$.
\end{lemma}

\begin{proof}
By definition, $q_{\mathbb T}(\varphi) =\inf_{\eta\in N_{\mathbb R}}\|\eta\star\varphi\|_1$ and since $a>1$ and $q_{\mathbb T}(\varphi)>0$, we have $a q_{\mathbb T}(\varphi)>q_{\mathbb T}(\varphi)$. The definition of the infimum therefore gives $\eta\in N_{\mathbb R}$ such that $\|\eta\star\varphi\|_1<a q_{\mathbb T}(\varphi)$.

The twisting operation is an action of the additive group $N_{\mathbb R}$. Hence, for every $\theta\in N_{\mathbb R}$,
\[
\begin{aligned}
 q_{\mathbb T}(\theta\star\varphi) &=\inf_{\zeta\in N_{\mathbb R}} \|\zeta\star(\theta\star\varphi)\|_1\\
 &=\inf_{\zeta\in N_{\mathbb R}} \|(\zeta+\theta)\star\varphi\|_1 =q_{\mathbb T}(\varphi),
\end{aligned}
\]
where the last equality follows because
$\zeta\mapsto\zeta+\theta$ is a bijection of $N_{\mathbb R}$.

Moreover, the proof of \cite[Theorem 8.19]{BJ26} gives $d_1^\NA(\xi\star\varphi,\xi'\star\varphi) \leq C\|\xi-\xi'\|$ for $\xi,\xi'\in N_{\mathbb R}$. Notice that the map $N_{\mathbb R}\ni\xi\longmapsto\|\xi\star\varphi\|_1$ is continuous, and as $N_{\mathbb Q}$ is dense in $N_{\mathbb R}$, we may choose $\xi\in N_{\mathbb Q}$ sufficiently close to $\eta$ that $\|\xi\star\varphi\|_1<a q_{\mathbb T}(\varphi)$ with $q_{\mathbb T}(\xi\star\varphi)=q_{\mathbb T}(\varphi)$.
\end{proof}

\begin{theorem}\label{thm:reduced-Kbeta-comparison}
Suppose that $G=\operatorname{Aut}^0(X,L)$ is reductive with maximal torus $\mathbb T\subset G$ and that $\operatorname{Fut}_{X,L}=0$. Then the following conditions are equivalent:
\begin{enumerate}[label=\textup{(\roman*)}]
\item $(X,L)$ is uniformly $\widehat{\mathrm K}$-polystable;
\item $(X,L)$ is uniformly reduced $\mathrm K^\beta$-polystable for every sufficiently large rational $\beta$;
\item $(X,L)$ is uniformly reduced $\mathrm K^\beta$-polystable for some rational $\beta>1$.
\end{enumerate}
\end{theorem}

\begin{proof}
We assume \textup{(i)}. By \cite[Corollary 8.18 and the proof of Theorem 8.19]{BJ26}, there exists $\sigma>0$ such that
\[
 \mathrm K^\NA(\psi)\geq\sigma q_{\mathbb T}(\psi)
\]
for every $\psi\in\mathcal E^{1,\NA}(L)^{\mathbb T}$.

We fix a $\mathbb T$-equivariant normal ample test-configuration metric $\varphi$. If $q_{\mathbb T}(\varphi)=0$, Lemma \ref{lem:reduced-Kbeta-basic} gives $K_{\mathbb T,\mathrm{red}}^\beta(\varphi)=0$, so the desired inequality is automatic.

Hence, we suppose that $q_{\mathbb T}(\varphi)>0$. By Lemma \ref{lem:optimal-rational-twist}, applied with $a=2$, there exists $\xi\in N_{\mathbb Q}$ such that $q_{\mathbb T}(\xi\star\varphi)=q_{\mathbb T}(\varphi)$ and $\|\xi\star\varphi\|_1<2q_{\mathbb T}(\varphi)$. We define $\psi:=\xi\star\varphi$. Since $\xi$ is rational, $\psi$ is the ample test-configuration metric associated to the rational product twist $\mathcal T_\xi$ of $\mathcal T$.

We now choose a $\mathbb T\times\mathbb G_m$-equivariant smooth dominant model with SNC central fibre, as in the proof of \cite[Corollary A]{Tru26}, and let $\psi_\beta$ be the log-discrepancy metric associated to $\psi$. The obstacle defining $\psi_\beta$ is $\mathbb T$-invariant. Since the action of $\mathbb T$ preserves the set of psh metrics lying below the obstacle, it also preserves their supremum. Thus $\psi_\beta\in\mathcal E^{1,\NA}(L)^{\mathbb T}$, and
\[
 \mathrm K^\NA(\psi_\beta) \geq\sigma q_{\mathbb T}(\psi_\beta).
\]

By Corollary \ref{cor:algebraic-polarisation-linear}, Lemma \ref{lem:norm-comparison}, and the choice of $\psi$, there is a constant $C_0>0$, independent of $\varphi$, such that
\[
 |\nabla_{-K_X}E^\NA(\psi)|
 \leq C_0J^\NA(\psi)
 <2C_0c_J^{-1}q_{\mathbb T}(\varphi).
\]
If $\mathrm K^\beta(\varphi_{\mathcal T_\xi})\geq\sigma q_{\mathbb T}(\varphi)$, the required lower bound already holds. Otherwise, Lemma \ref{lem:distance-entropy} gives
\[
 d_1^\NA(\psi,\psi_\beta)
 \leq\frac{\sigma+2C_0c_J^{-1}}{\beta}
 q_{\mathbb T}(\varphi).
\]
Furthermore, \eqref{eq:reduced-norm-Lipschitz} gives
\[
 q_{\mathbb T}(\psi_\beta)
 \geq q_{\mathbb T}(\psi)-d_1^\NA(\psi,\psi_\beta)
 =q_{\mathbb T}(\varphi)-d_1^\NA(\psi,\psi_\beta),
\]
while $J^\NA(\psi)<2c_J^{-1}q_{\mathbb T}(\varphi)$. Theorem \ref{thm:mabuchi-comparison} therefore yields
\[
\begin{aligned}
 \mathrm K^\beta(\varphi_{\mathcal T_\xi})
 &\geq\sigma q_{\mathbb T}(\psi_\beta) -C\,d_1^\NA(\psi,\psi_\beta)^{1/2}J^\NA(\psi)^{1/2}\\
 &\geq\left(\sigma-
 \frac{\sigma(\sigma+2C_0c_J^{-1})}{\beta} -C\left(\frac{2c_J^{-1}(\sigma+2C_0c_J^{-1})}{\beta}\right)^{1/2}\right)
 q_{\mathbb T}(\varphi).
\end{aligned}
\]
Defining 
\[\varepsilon_\beta:=\frac{\sigma(\sigma+2C_0c_J^{-1})}{\beta}+C\left(\frac{2c_J^{-1}(\sigma+2C_0c_J^{-1})}{\beta}\right)^{1/2}\]
we have $\varepsilon_\beta\geq 0$ and $\varepsilon_\beta\to 0$. Thus we can find $\beta_0$ such that for all rational $\beta\geq \beta_0$ we have 
\[\mathrm K^\beta(\varphi_{\mathcal T_\xi})\geq \frac{\sigma}{2}q_{\mathbb T}(\varphi). \]
Taking the supremum over all rational product twists gives
\[
 \mathrm K_{\mathbb T,\mathrm{red}}^\beta(\varphi) \geq\frac\sigma2q_{\mathbb T}(\varphi),
\]
which proves \textup{(ii)}. The implication \textup{(ii)}$\Rightarrow$\textup{(iii)} is immediate.

It remains to prove \textup{(iii)}$\Rightarrow$\textup{(i)}. We fix $\beta>1$ and $\gamma>0$ as in \textup{(iii)}. For every $\mathbb T$-equivariant normal ample test-configuration metric, Lemma \ref{lem:reduced-Kbeta-basic} gives
\begin{equation}\label{eq:reduced-ample-input}
 \mathrm K^\NA(\varphi) \geq\mathrm K_{\mathbb T,\mathrm{red}}^\beta(\varphi) \geq\gamma q_{\mathbb T}(\varphi).
\end{equation}

We first pass this inequality to $\mathbb T$-equivariant model metrics. After taking a $\mathbb T\times\mathbb G_m$-equivariant smooth SNC resolution, \cite[Corollary A and its proof]{Tru26}, together with \cite[Lemma 4.5]{Li21}, provides $\mathbb T$-equivariant smooth ample metrics $\varphi_m$ such that $d_1^\NA(\varphi_m,\varphi)\longrightarrow0$ and $\Ent^\NA(\varphi_m)\longrightarrow\Ent^\NA(\varphi)$, i.e. that $\mathrm K^\NA(\varphi_m)\longrightarrow\mathrm K^\NA(\varphi)$ and $q_{\mathbb T}(\varphi_m)\longrightarrow q_{\mathbb T}(\varphi)$. Passing to the limit in \eqref{eq:reduced-ample-input} proves the same uniform inequality for all $\mathbb T$-equivariant model metrics.

Since $G$ is reductive and $T$ is maximal, the identity component of the centraliser of $T$ is $\mathbb T$. The model inequality obtained above is therefore precisely $\mathbb T$-uniform K-stability for models. By \cite[Lemma 8.17]{BJ26}, this gives uniform $T$-equivariant $\widehat{\mathrm K}$-polystability, and \cite[Corollary 8.18]{BJ26} then gives uniform $\widehat{\mathrm K}$-polystability.
\end{proof}

\begin{remark}
By \cite[Theorem A]{BJ26}, the conditions in Theorem \ref{thm:reduced-Kbeta-comparison} are also equivalent to the existence of a cscK metric. It is closely related to the reduced Yau--Tian--Donaldson correspondence of \cite[Theorem 10.4]{DZ26}.
\end{remark}

\subsection{Semistability thresholds}
\label{sec:k-ss}

We conclude by comparing K-semistability with $\mathrm K^\beta$-semistability.

For every normal ample test configuration $\mathcal T$, we define
\[
 \lambda_\beta(X,L) :=\inf_{\|\varphi_{\mathcal T}\|_1>0} \frac{\mathrm K^\beta(\varphi_{\mathcal T})}{\|\varphi_{\mathcal T}\|_1}, \qquad \lambda_\infty(X,L) :=\inf_{\|\varphi_{\mathcal T}\|_1>0} \frac{\mathrm K^\NA(\varphi_{\mathcal T})}{\|\varphi_{\mathcal T}\|_1}.
\]

If $\|\varphi_{\mathcal T}\|_1=0$, then $\varphi_{\mathcal T}$ is constant by Lemma \ref{lem:zero-norm}. The translation invariance of the algebraic formulas gives $\mathrm K^\beta(\varphi_{\mathcal T})=\mathrm K^\NA(\varphi_{\mathcal T})=0$. Consequently, $\mathrm K^\beta$-semistability is equivalent to $\lambda_\beta\geq0$, while K-semistability is equivalent to $\lambda_\infty\geq0$ by \cite[Proposition 8.2]{BHJ17}.

Since $X$ is smooth, we have $\Ent_{\mathcal T}^\beta\geq0$. Corollary \ref{cor:algebraic-polarisation-linear} and Lemma \ref{lem:norm-comparison} therefore give a constant $C_0>0$ such that
\[
 -C_0\leq\lambda_\beta\leq\lambda_\infty<+\infty.
\]

\begin{lemma}\label{lem:threshold-models}
Every log-discrepancy metric $\varphi_\beta$ satisfies
\begin{equation}\label{eq:model-optimal-lower-bound}
 \mathrm K^\NA(\varphi_\beta) \geq\lambda_\infty(X,L)\|\varphi_\beta\|_1.
\end{equation}
\end{lemma}

\begin{proof}
By the definition of $\lambda_\infty$, the inequality
\[
 \mathrm K^\NA(\psi)
 \geq\lambda_\infty(X,L)\|\psi\|_1
\]
holds for every normal ample test-configuration metric $\psi$, including those of zero norm. Proposition \ref{prop:input}, applied with $\Lambda=\lambda_\infty(X,L)$, proves the assertion.
\end{proof}

\begin{theorem}\label{thm:threshold-convergence}
Suppose that $(X,L)$ is K-semistable. Then there exists $\varepsilon_\beta\geq 0$, with $\varepsilon_\beta\to 0$ as $\beta \to \infty$ such that
\begin{equation}\label{eq:threshold-error}
 \lambda_\infty-\varepsilon_\beta
 \leq\lambda_\beta\leq\lambda_\infty.
\end{equation}
In particular, $\lambda_\beta(X,L)\nearrow\lambda_\infty(X,L)$ as $\beta\to\infty$.
\end{theorem}

\begin{proof}
Since $(X,L)$ is K-semistable, $\lambda_\infty(X,L)\geq0$. Let $\mathcal T$ be a normal ample test configuration with associated metric $\varphi_{\mathcal T}$ satisfying $\|\varphi_{\mathcal T}\|_1>0$, and let $\varphi_\beta$ be its log-discrepancy metric. By Lemma \ref{lem:threshold-models},
\[
 \mathrm K^\NA(\varphi_\beta)
 \geq\lambda_\infty(X,L)\|\varphi_\beta\|_1.
\]
The hypothesis of Corollary \ref{cor:transfer} therefore holds with $\Lambda=\lambda_\infty(X,L)$. It follows that there exists $\varepsilon_\beta\geq0$, independent of $\mathcal T$, such that $\varepsilon_\beta=O(\beta^{-1/2})$ and
\[
 \mathrm K^\beta(\varphi_{\mathcal T})
 \geq\bigl(\lambda_\infty(X,L)-\varepsilon_\beta\bigr)
 \|\varphi_{\mathcal T}\|_1.
\]
Dividing by $\|\varphi_{\mathcal T}\|_1>0$ and taking the infimum over all normal ample test configurations of positive norm yields
\[
 \lambda_\beta(X,L) \geq\lambda_\infty(X,L)-\varepsilon_\beta.
\]

On the other hand, \cite[Theorem 9.5]{DZ26} gives $\mathrm K^\beta(\varphi_{\mathcal T})\leq\mathrm K^\NA(\varphi_{\mathcal T})$, and hence $\lambda_\beta(X,L)\leq\lambda_\infty(X,L)$. The same theorem shows that $\mathrm K^\beta(\varphi_{\mathcal T})$ is increasing in $\beta$ for every $\mathcal T$, so $\lambda_\beta(X,L)$ is also increasing. Since $\varepsilon_\beta\to0$, the result follows.
\end{proof}

\begin{corollary}\label{cor:asymptotic-semistability}
There exists $C'=C'(X,L)>0$ such that every K-semistable $(X,L)$ satisfies $\mathrm K^\beta(\varphi_{\mathcal T})\geq-C'\beta^{-1/2}\|\varphi_{\mathcal T}\|_1$ for every normal ample test configuration and every rational $\beta>1$. Moreover:
\begin{enumerate}[label=\textup{(\roman*)}]
\item $(X,L)$ is K-semistable if and only if
$\liminf_{\beta\to\infty,\,\beta\in\mathbb Q}\lambda_\beta(X,L)\geq0$;
\item if $(X,L)$ is K-semistable but not uniformly K-stable, then $-C'\beta^{-1/2}\leq\lambda_\beta(X,L)\leq0$ and $\lambda_\beta(X,L)\nearrow0$;
\item if $(X,L)$ is uniformly K-stable, then
$\lambda_\beta(X,L)>0$ for all sufficiently large rational $\beta$.
\end{enumerate}
\end{corollary}

\begin{proof}
Theorem \ref{thm:threshold-convergence} gives $\lambda_\beta\geq\lambda_\infty-O(\beta^{-1/2})$. Under K-semistability, $\lambda_\infty\geq0$, so after increasing the fixed constant this implies $\lambda_\beta\geq-C'\beta^{-1/2}$, which also gives the first displayed inequality.

For \textup{(i)}, K-semistability implies $\lambda_\beta\to\lambda_\infty\geq0$. Conversely, $\lambda_\beta\leq\lambda_\infty$ for every rational $\beta>1$, so a nonnegative lower limit forces $\lambda_\infty\geq0$, which is equivalent to K-semistability.

For \textup{(ii)}, K-semistability gives
$\lambda_\infty(X,L)\geq0$. Suppose that
$\lambda_\infty(X,L)>0$. Every normal ample test configuration $\mathcal T$ satisfies $\mathrm K^\NA(\varphi_{\mathcal T}) \geq\lambda_\infty(X,L)\|\varphi_{\mathcal T}\|_1$. By Lemma \ref{lem:norm-comparison},
\[
 \mathrm K^\NA(\varphi_{\mathcal T})
 \geq c_J\lambda_\infty(X,L) J^\NA(\varphi_{\mathcal T}).
\]
Since $c_J\lambda_\infty(X,L)>0$, $(X,L)$ is uniformly K-stable by \cite[Proposition 8.2]{BHJ17}, contrary to the assumption. Therefore $\lambda_\infty(X,L)=0$.

Theorem \ref{thm:threshold-convergence} now gives $\varepsilon_\beta \to 0$ with $-\varepsilon_\beta \leq\lambda_\beta(X,L)\leq0$, i.e. $\lambda_\beta(X,L)\nearrow0$. Using $\varepsilon_\beta=O(\beta^{-1/2})$ and increasing the constant if necessary, we obtain
\[
 -C'\beta^{-1/2} \leq\lambda_\beta(X,L)\leq0.
\]

For \textup{(iii)}, uniform K-stability gives a constant $\delta>0$
such that $\mathrm K^\NA(\varphi_{\mathcal T}) \geq\delta J^\NA(\varphi_{\mathcal T}) \geq\frac{\delta}{C_J}\|\varphi_{\mathcal T}\|_1$. Hence $\lambda_\infty(X,L)\geq\frac{\delta}{C_J}>0$. Since $\lambda_\beta\to\lambda_\infty>0$, we obtain $\lambda_\beta>0$ for all sufficiently large rational $\beta$.
\end{proof}

\begin{definition}[Asymptotic $\mathrm K^\beta$-semistability]\label{def:asymptotic-Kbeta-semistability}
We say that $(X,L)$ is \emph{asymptotically $\mathrm K^\beta$-semistable} if
\[
 \liminf_{\beta\to\infty,\,\beta\in\mathbb Q}
 \lambda_\beta(X,L)\geq0.
\]
\end{definition}

Corollary \ref{cor:asymptotic-semistability}\textup{(i)} says precisely that ordinary K-semistability is equivalent to asymptotic $\mathrm K^\beta$-semistability. Since $\widehat{\mathrm K}$-semistability implies ordinary K-semistability by restriction, equivalently by \cite[Proposition 8.13(i), and Corollary 8.15]{BJ26}, it also implies
\[
 \widehat{\mathrm K}\text{-semistability}
 \Longrightarrow
 \text{asymptotic }\mathrm K^\beta\text{-semistability}.
\]

\end{document}